\RequirePackage{tikz}
\documentclass[sn-mathphys]{sn-jnl}

\usepackage[utf8]{inputenc}
\usepackage{dsfont}
\usepackage{transparent}
\usepackage{placeins}
\usepackage[nameinlink]{cleveref}

\renewcommand{\R}{\mathbb{R}}
\newcommand{\N}{\mathbb{N}}
\newcommand{\U}{\mathcal{U}}
\newcommand{\X}{\mathcal{X}}

\renewcommand{\P}{\mathcal{P}}
\newcommand{\sX}{{_\X}}
\newcommand{\sU}{{_\U}}
\newcommand{\ddes}{{d_{\textnormal{o}}}}
\newcommand{\dran}{{d_{\textnormal{r}}}}
\renewcommand{\e}{{\varepsilon}}

\newcommand{\dd}{\mathrm d}

\DeclareMathOperator*{\argmin}{arg\,min}

\DeclareMathOperator{\supp}{supp}
\DeclareMathOperator{\dist}{dist}

\jyear{2021}%

\theoremstyle{thmstyleone}%
\newtheorem{theorem}{Theorem}[section]
\newtheorem{assumption}{Assumption}[section]
\Crefname{assumption}{Assumption}{Assumptions}
\newtheorem{lemma}{Lemma}[section]

\theoremstyle{thmstyletwo}%
\newtheorem{remark}{Remark}[section]%

\theoremstyle{thmstylethree}%
\newtheorem{definition}{Definition}[section]%

\begin{document}

\title[The Continuous Stochastic Gradient Method]{The Continuous Stochastic Gradient Method}
\subtitle{Part I - Convergence Theory}


\author[1]{\fnm{Max} \sur{Grieshammer}}\email{max.grieshammer@fau.de}

\author[1,2]{\fnm{Lukas} \sur{Pflug}}\email{lukas.pflug@fau.de}

\author[1]{\fnm{Michael} \sur{Stingl}}\email{michael.stingl@fau.de}
\author*[1]{\fnm{Andrian} \sur{Uihlein}}\email{andrian.uihlein@fau.de}

\affil[1]{\orgdiv{Department of Mathematics, Chair of Applied Mathematics}, \orgname{Friedrich-Alexander-Universität Erlangen-Nürnberg (FAU)}}
\affil[2]{\orgdiv{Competence Unit for Scientific Computing}, \orgname{Friedrich-Alexander-Universität Erlangen-Nürnberg (FAU)}}




\abstract{In this contribution, we present a full overview of the \textit{continuous stochastic gradient} (CSG) method, including convergence results, step size rules and algorithmic insights. We consider optimization problems in which the objective function requires some form of integration, e.g., expected values. Since approximating the integration by a fixed quadrature rule can introduce artificial local solutions into the problem while simultaneously raising the computational effort, stochastic optimization schemes have become increasingly popular in such contexts. However, known stochastic gradient type methods are typically limited to expected risk functions and inherently require many iterations. The latter is particularly problematic, if the evaluation of the cost function involves solving multiple state equations, given, e.g., in form of partial differential equations.
To overcome these drawbacks, a recent article introduced the CSG method, which reuses old gradient sample information via the calculation of design dependent integration weights to obtain a better approximation to the full gradient. 

While in the original CSG paper convergence of a subsequence was established for a diminishing step size, here, we provide a complete convergence analysis of CSG for constant step sizes and an Armijo-type line search. Moreover, new methods to obtain the integration weights are presented, extending the application range of CSG to problems involving higher dimensional integrals and distributed data.
}

\maketitle
\bmhead{Acknowledgments}
The research was funded by the Deutsche Forschungsgemeinschaft (DFG, German Research Foundation) - Project-ID 416229255 - CRC 1411).

\section{Introduction}
In this contribution, we present a full overview of the \textit{continuous stochastic gradient} (CSG) method, including convergence results, step size rules, algorithmic insights and applications from topology optimization. The prototype idea for CSG was first proposed in \cite{pflug_CSG}. Therein, it was shown that for expected-valued objective functions, CSG with diminishing step sizes and exact integration weights (see \Cref{sec:Weights}) almost surely produces a subsequence converging to a stationary point \cite[Theorem 20]{pflug_CSG}. This work severely generalizes this result, providing proofs for convergence of the full sequence of iterates in the case of constant step sizes and backtracking line-search techniques. Additionally, the convergence results hold in a less restrictive setting and for generalized approaches to the integration weight calculation (see \Cref{sec:Weights}).  Before going into details, we want to explain the reason for introducing what seems like yet another first-order stochastic optimization scheme.
\subsection{Motivation from PDE-constrained Optimization}
Within PDE-constrained optimization, settings with expected-valued objective functions arise in numerous applications, ranging from purely stochastic settings, like machine learning or noisy simulations \cite{Intro01,Intro04}, to fully deterministic problems, in which one is interested in a design that is optimal for an infinite set of outer parameters, e.g. \cite{Intro02,Intro03,semmler2015shape,singh2022robust}. Especially in large scale settings, one usually does not consider deterministic approaches (see, e.g., \cite{wright1999numerical}) for the solution of such problems, as they are generally too computationally expensive or even intractable. Instead, one uses stochastic optimization schemes, like the \textit{Stochastic Gradient} (SG) \cite{Monro1951} or \textit{Stochastic Average Gradient} (SAG) method \cite{LeRoux2017}. A large number of schemes have been derived from these and thoroughly analyzed, including \textit{sequential quadratic programming for stochastic optimization} \cite{curtis2021worst,berahas2021sequential}, \textit{quasi-Newton stochastic gradient schemes} \cite{bordes2009sgd,pilanci2017newton,byrd2016stochastic,moritz2016linearly} and the well known adaptive gradient schemes \textit{Adam \& Adagrad} \cite{kingma2014adam,duchi2011adaptive}, to name some prominent examples.

Problematically, such methods rely on a heavily restrictive setting, in which the objective function value of a design $u$ is simply given as the expected value of some quantity $j$, i.e., $J(u) = \mathbb{E}_x[j(u,x)]$. Even the basic setting of nesting two expectation values, i.e., $J(u) = \mathbb{E}_y[j_2(y,\mathbb{E}_x[(j_1(u,x)])]$, is beyond the scope of the mentioned schemes and requires special techniques, e.g. the \textit{Stochastic Composition Gradient Descent} (SCGD) method \cite{SCGDPaper}, which itself is again only applicable in this specific setting.

In this contribution, we investigate an application from the field of optimal nanoparticle design, which admits exactly such a complex structure. Our main interest lies in the optical properties of nanoparticles. Specifically, the color of particulate products, which has been of great interest for many fields of research \cite{Color1,Color2,Color3,Color4,Color5,PAMM}, is what we are trying to optimize for. This and similar applications serve only as motivation in this paper. However, in the numerical analysis of CSG \cite{CSGPart2}, it is demonstrated that CSG indeed represents an efficient approach to such problems. While a detailed introduction to this setting is given in \cite{CSGPart2}, we want to briefly summarize the problems arising in this application. 

To obtain the color of a particulate product, we need to calculate the important optical properties of the nanoparticles in the product. For each particle design, this requires solving the time-harmonic Maxwell's equations, which, depending of the setting, is numerically expensive. Furthermore, the color of the whole product is not determined by the optical properties of a single particle. Instead, we need to average these properties over, e.g., the particle design distribution and their orientation in the particulate product. Afterwards, the averaged values are used to calculate intermediate results for the special setting. These results then need to be integrated over the range of wavelengths, which are visible to the human eye, to obtain the color of the particulate product. Finally, the objective function uses the resulting color in a nonlinear fashion, before yielding the actual objective function value. In general, the objective function has the form
\begin{equation*}
    J(u) = j_3\big(u, \mathbb{E}_y\big[ j_2(u,y,\mathbb{E}_x[j_1(u,x,y)])\big]\big)
\end{equation*}
and can easily involve even more convoluted terms in more advanced settings.

On the one hand, the computational cost of deterministic approaches to such problems range from tremendous to infeasible, since $j_1$ is typically not easy to evaluate. On the other hand, standard schemes from stochastic optimization, like the ones mentioned above, simply cannot solve the full problem. Thus, we are in the need for a general method to tackle optimization problems, which are given by arbitrary concatenations of nonlinear functions and expectation values.
\subsection{Properties of CSG}
As mentioned in the previous section, the CSG method aims to offer a new approach to optimization problems that involve integration of properties, which are expensive to evaluate. In each iteration, CSG draws a very small number (typically 1) of random samples, as is the case for SG. However, instead of discarding the information collected through these evaluations after the iteration, these results are stored. In later iterations, all of the information collected along the way is used to build an approximation to the full objective function and its gradient, by a special linear combination. The weights appearing in this linear combination, which we call \textit{integration weights}, can be calculated in several different fashions, which are detailed in \Cref{sec:Weights}.

As a key result for CSG, we are able to show that the approximation error in both the gradient and the objective function approximation vanishes during the optimization process (\Cref{lem:GradientError}). Thus, in the course of the iterations, CSG more and more behaves like an exact full gradient algorithm and is able to solve optimization problems far beyond the scope of standard SG-type methods. Furthermore, we show that this special behaviour results in CSG having convergence properties similar to full gradient descent methods, while keeping advantages of stochastic optimization approaches, i.e., each step is computationally cheap. To be precise, we prove convergence to a stationary point for constant step sizes (\Cref{theo:ConstantStepConvergence}) and an Armijo-type line search (\Cref{theo:ArmijoConvergence}), which is based on the gradient and objective function approximations of CSG.
\subsection{Limitations of the Method}
While CSG combines advantages of deterministic and stochastic optimization schemes, the hybrid approach also yields drawbacks, which we try to address throughout this contribution. As mentioned earlier, the intended application for CSG lies in expected-valued optimization problems, in which solving the state problem is computationally expensive. In many other settings that heavily rely on stochastic optimization methods, e.g., neural networks, the situation is different. Here, we can efficiently obtain a stochastic (sub-)gradient, meaning that we are better off simply performing millions of SG iterations, than a few thousand CSG steps. In these situations, the inherent additional computational effort that lies within the calculation of the CSG integration weights (see \Cref{sec:Weights}) is no longer negligible and usually can not be compensated by the improved gradient approximation.

Furthermore, the convergence rate of CSG worsens as the dimension of integration increases. While this can be avoided, if the objective function imposes additional structure, it remains a drawback in general. A detailed analysis of this issue can be found in \cite{CSGPart2}.

We emphasize that CSG and SG-type methods share many similarities. However, their intended applications are complementary to each other, with SG preferring objective functions of simple structure and computationally cheap sampling, while CSG prefers the opposite.  
\subsection{Structure of the Paper}
\Cref{sec:Assumptions} states the general framework of this contribution as well as the basic assumptions we impose throughout the paper. Furthermore, the general CSG method is presented. 

The integration weights, which play an important role in the CSG scheme, are detailed in \Cref{sec:Weights}. Therein, we introduce four different methods to obtain weights which satisfy the necessary assumptions made in \Cref{sec:Assumptions} and analyze some of their properties. The section also describes techniques to implement mini-batches in the CSG method. 

Auxiliary results concerning CSG are given in \Cref{sec:Auxiliary}. This includes the gradient approximation property of CSG (\Cref{lem:GradientError}) and a numerical example for the generalized setting of the optimization problem (\Cref{subsec:ExampleGeneralizedSetting}). The first part of the convergence theory, i.e., convergence for constant steps (\Cref{theo:ConstantStepConvergence} and \Cref{cor:ConstantStepConvergence}), is presented in \Cref{sec:ConvergenceConstant} and tested on a simple example (\Cref{subsec:ExampleConstSteps}). 

Afterwards, in \Cref{sec:ConvergenceArmijo}, we incorporate an Armijo-type line search in the CSG method and provide a convergence analysis for the resulting optimization scheme (\Cref{theo:ArmijoConvergence}). The theoretical results are additionally tested for an academic example (\Cref{sec:StepsizeStabilityEx}).

A numerical analysis of CSG, concerning the performance for non-academic examples and convergence rates, is not part of this contribution, as this can be found in \cite{CSGPart2}.

\section{Setting and Assumptions}\label{sec:Assumptions}
In this section, we introduce the general setting and formulate the assumptions made throughout this contribution. Additionally, the basic CSG algorithm is presented and some preliminary results are stated.
\subsection{Setting}\label{subsec:SettingAndAssumptions}
As mentioned above, the convergence analysis of the CSG method is carried out in a simplified setting in order to shorten notation and improve the overall readability. For the general case, see \Cref{rem:GeneralizedSetting}.
\begin{definition}[Objective Function]
For $\ddes,\dran\in\N$, we introduce the set of admissible optimization variables $\U\subset\R^\ddes$ and the parameter set $\X\subset\R^\dran$. The objective function $J:\U\to\R$ is then given by
\begin{equation*}
    J(u):= \mathbb{E}\left[ j(u,X)\right] = \int_\X j(u,x)\mu(\dd x),
\end{equation*}
where we assume $j\in C^1(\U\times\X\,;\,\R)$ to be measurable and $X\sim\mu$.

The (simplified) optimization problem is then given by
\begin{equation}
    \min_{u\in\U} \quad \int_\X j(u,x)\mu(\mathrm{d}x). \label{eq:Problemstellung}
\end{equation}
\end{definition}
Since $\U$ and $\X$ are finite dimensional, we do not have to consider specific norms on these spaces due to equivalence of norms and can instead choose them problem specific. In the following, we will denote them simply by $\Vert\cdot\Vert_\sU$ and $\Vert\cdot\Vert_\sX$.

During the optimization, we need to draw independent random samples of the random variable $X$, as stated in the following assumption.

\begin{assumption}[Sample Sequence] \label{assum:SampleSequence}
The sequence of samples $(x_n)_{n\in\N}$ is a sequence of independent identically distributed realizations of the random variable $X\sim\mu$.
\end{assumption}
\begin{remark}[Almost Sure Density]\label{rem:SequenceDense}
We define the support of the measure $\mu$ as the (closed) subset
\begin{equation*}
     \supp(\mu) := \left\{ x\in\X\,:\,\mu\left(\mathcal{B}_\varepsilon(x)\right)>0\text{ for all }\varepsilon>0\right\}\subseteq \R^\dran.
\end{equation*}
It is important to note, that a sample sequence satisfying \Cref{assum:SampleSequence} is dense in $\supp(\mu)$ with probability 1. This can be seen as follows:

Let $x\in\supp(\mu)$ and $\varepsilon>0$. Then, given an independent identically distributed sequence $x_1,x_2,\ldots\sim\mu$, we have
\begin{equation*}
    \mathbb{P}\big(x_n\not\in\mathcal{B}_\varepsilon(x)\big) = 1-\mu(\mathcal{B}_\varepsilon(x)) < 1.
\end{equation*}
Hence, by the Borel-Cantelli Lemma \cite[Theorem 2.7]{klenke2013probability},
\begin{equation*}
    \mathbb{P}\big( x_n\not\in\mathcal{B}_\varepsilon(x)\text{ for all }n\in\N\big) = 0.
\end{equation*}
Thus, the sequence $(x_n)_{n\in\N}$ is dense in $\supp(\mu)$ with probability 1.
\end{remark}
\begin{remark}[Almost sure convergent results]
The (almost sure) density of the sample sequence plays a crucial role in the upcoming proofs. Hence, the convergence results presented in this contribution all hold in the almost sure sense. However, to improve the readability, we refrain from always mentioning this explicitly.
\end{remark}
Moreover, the sets $\U$ and $\X$ need to satisfy the following regularity conditions:
\begin{assumption}[Regularity of $\U$, $\X$ and $\mu$]\label{assum:Nummer2}
The set $\U\subset\R^\ddes$ is compact and convex. The set $\X\subset\R^\dran$ is open and bounded with $\supp(\mu)\subset\X$.
\end{assumption}
Notice that the second part of \Cref{assum:Nummer2} can always be achieved, as long as $\supp(\mu)\subset\R^\dran$ is bounded.

Finally, as in the deterministic case, we assume the gradient of the objective function to be Lipschitz continuous, in order to obtain convergence for constant step sizes.
\begin{assumption}[Regularity of $j$]\label{assum:RegularitySmallj}
The function $\nabla_1 j : \U\times\X\to\R^\ddes$ is bounded and Lipschitz continuous, i.e., there exists constants $C,L_j\in\R_{>0}$ such that
\begin{align*}
    \Vert \nabla_1 j(u,x)\Vert &\le C, \\
    \Vert \nabla_1 j(u_1,x_1) - \nabla_1 j(u_2,x_2)\Vert &\le L_j \big( \Vert u_1-u_2\Vert_\sU + \Vert x_1-x_2\Vert_\sX\big)
\end{align*}
for all $u,u_1,u_2\in\U$ and $x,x_1,x_2\in\X$.
\end{assumption}
\begin{remark}[Regularity of $\nabla J$]\label{rem:Nummer3} By \Cref{assum:SampleSequence,assum:Nummer2,assum:RegularitySmallj}, $J:\U\to\R$ is $L$-smooth, i.e., there exists $L>0$ such that
\begin{equation*}
    \left\Vert \nabla J(u_1)-\nabla J(u_2)\right\Vert \le L\Vert u_1-u_2\Vert_\sU \quad\text{for all }u_1,u_2\in\U.
\end{equation*}
\end{remark}

\subsection{The CSG Method}
Given a starting point $u_1\in\U$ and a random parameter sequence $x_1,x_2,\ldots$ according to \Cref{assum:SampleSequence}, the basic CSG method is stated in \Cref{alg:GeneralCSG}. In each iteration, the inner objective function $j$ and gradient $\nabla_1 j$ are evaluated only at the current design-parameter-pair $(u_n,x_n)$. Afterwards, the integrals appearing in $J$ and $\nabla J$ are approximated by a linear combination, consisting of all samples accumulated in previous iterations.

The coefficients $(\alpha_k)_{k=1,\ldots,n}$ appearing in Step 4 of \Cref{alg:GeneralCSG}, which we call \textit{integration weights}, are what differentiates CSG from other stochastic optimization schemes. In \cite{pflug_CSG}, where the main idea of the CSG method was proposed for the first time, a special choice of how to calculate these weights was already presented. A recap of this procedure as well as several new methods to obtain the integration weights is given in detail in \Cref{sec:Weights}.

Furthermore, $\P_\U$ appearing in Step 8 of \Cref{alg:GeneralCSG} denotes the orthogonal projection (in the sense of $\Vert\cdot\Vert_\U)$, i.e.,
\begin{equation*}
    \P_\U(u) := \argmin_{\bar{u}\in\U}\Vert u-\bar{u}\Vert_\sU.
\end{equation*}
The final general assumption we will use throughout this contribution is related to the integration weights mentioned above.
\begin{assumption}[Integration Weights]\label{assum:Weights}
Denote the sequence of designs and random parameters generated by \Cref{alg:GeneralCSG} until iteration $n\in\N$ by $(u_k)_{k=1,\ldots,n}$ and $(x_k)_{k=1,\ldots,n}$, respectively. Define $k^n:\X\to \{1,\ldots,n\}$ as
\begin{equation}
    k^n(x) := \argmin_{k=1,\ldots,n} \big( \Vert u_n-u_k\Vert_\sU + \Vert x - x_k\Vert_\sX\big).\label{eq:khochn}
\end{equation}
Then, for all $n\in\N$, there exists a probability measure $\mu_n$ on $\X$ such that
\begin{equation}
    \alpha_k^{(n)} = \int_\X \delta_{k^n(x)}(k)\mu_n(\mathrm{d}x)\quad\text{for all }k=1,\ldots,n.\label{eq:assumMeasure1}
\end{equation}
Here, $\big(\alpha_k^{(n)}\big)_{k=1,\ldots,n}$ denotes the integration weights in iteration $n\in\N$, while $\delta_{k^n(x)}(k)$ corresponds to the Dirac measure of $k^n(x)$.

Furthermore, the measures $(\mu_n)_{n\in\N}$ converge weakly to $\mu$, i.e., $\mu_n\Rightarrow \mu$ for ${n\to\infty}$, where
\begin{equation*}
    \mu_n\Rightarrow \mu \quad \text{iff} \quad \int_\X f(x)\mu_n(\mathrm{d}x) \to \int_\X f(x)\mu(\mathrm{d}x)\quad \text{for all } f\in\mathcal{C}_b(\X,\R).
\end{equation*}
\end{assumption}
There is a very simple idea hidden behind the technicalities of \Cref{assum:Weights}. Condition \eqref{eq:assumMeasure1} states that the integration weights should be somehow based on a nearest neighbor approximation
\begin{equation*}
    \nabla_1 j(u_n,x) \approx \nabla_1 j\big(u_{k^n(x)}, x_{k^n(x)}\big),
\end{equation*}
while the condition $\mu_n\Rightarrow\mu$ ensures that the weight of a sample is reasonably chosen, i.e.,
\begin{equation*}
    \alpha_k^{(n)} = \int_\X \delta_{k^n(x)}(k)\mu_n(\mathrm{d}x) \approx \int_\X \delta_{k^n(x)}(k)\mu(\mathrm{d}x).
\end{equation*}
Due to the finite dimensional setting, all convergence results proven in this contribution hold independent of the chosen norm on $\U\times\X$ implied by \eqref{eq:khochn}. However, the specific choice may strongly influence the behavior (see, \Cref{rem:MetricWeights}) and performance of CSG. Further insight on the integration weights and multiple methods to obtain weights satisfying \Cref{assum:Weights} are given in the following \Cref{sec:Weights}.
\begin{algorithm} 
\caption{General CSG Method}
\label{alg:GeneralCSG}
\begin{algorithmic}[0]
\While{Termination condition not met}
\State{Sample objective function (optional):\\
    $\qquad j_n := j(u_n,x_n)$}
\State{Sample gradient:\\
    $\qquad g_n := \nabla_1 j(u_n,x_n)$}
\State{Calculate integration weights \\ $\qquad \alpha_k$}
\State{Calculate search direction:\\
    $\qquad \hat G_{n} := \sum_{k=1}^n  \alpha_k g_k \vphantom{\Big(}$}
\State{Compute objective function value approximation (optional):\\
    $\qquad \hat J_{n} := \sum_{k=1}^n  \alpha_k j_k \vphantom{\Big(}$} 
\State{Choose step size \\ $\qquad\tau_n$}
\State{Gradient step:\\
    $\qquad u_{n+1} := \P_\U\big(u_n - \tau_n \hat G_{n}\big) \vphantom{\Big(}$}
\State{Update index:\\
    $\qquad n \leftarrow n+1 \vphantom{\Big(}$}
\EndWhile
\end{algorithmic}
\rule{\textwidth}{.75pt}
The general CSG method as proposed in \cite{pflug_CSG} for the simplified setting \eqref{eq:Problemstellung}. Further information on how to carry out the integration weight calculation in practice is given in \Cref{sec:Weights}.   
\end{algorithm}
\section{Integration Weights}\label{sec:Weights}
We start with a brief motivation: Suppose that we are currently in the 9th iteration of the CSG algorithm. So far, we have sampled $\nabla_1 j(u,x)$ at nine points $(u_i,x_i)_{i=1,\ldots,9}$ and the full gradient at the current design is given by 
\begin{equation*}
    \nabla J (u_9) = \int_\X \nabla_1 j(u_9,x)\mu(\mathrm{d}x).
\end{equation*}
In \Cref{fig:WeightsNothing}, the situation is shown for the case $\dran = \ddes = 1$. How can we use the samples $\big(\nabla_1 j(u_i,x_i)\big)_{i=1,\ldots,9}$ in an optimal fashion, to build an approximation to $\nabla J(u_9)$?

We will present four different methods of calculating the integration weights $(\alpha_i)_{i=1,\ldots,n}$, each with their own benefits and downsides.
\begin{figure}
  \begin{minipage}[c]{0.5\textwidth}
    \includegraphics[width=\textwidth]{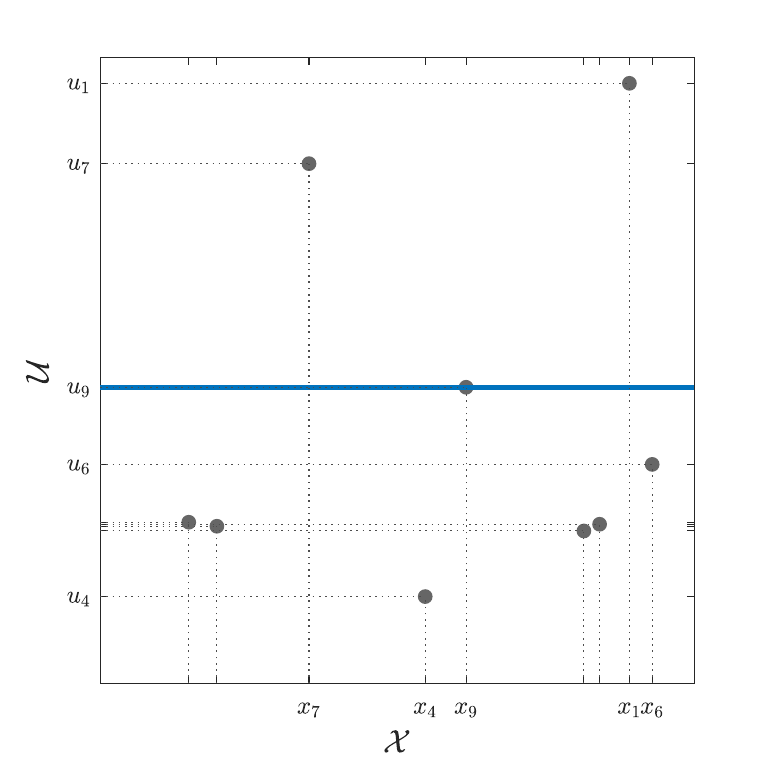}
  \end{minipage}\hfill
  \begin{minipage}[c]{0.45\textwidth}
    \caption{The grey dots represent our nine sample points $(u_i,x_i)\in\U\times\X$. The full gradient of the objective function is obtained by integrating $\nabla_1 j(u_9,\cdot)$ along the blue line.}
    \label{fig:WeightsNothing}
  \end{minipage}
\end{figure}

\subsection{Exact Weights}
To approximate the value of the integral along the bold line, we may use a nearest neighbor approximation. The underlying idea is visualized in \Cref{fig:ExactWeights}. Thus, we define the sets
\begin{align*}
    M_k := \big\{ x\in\X \, : \, \Vert u_n& - u_k \Vert_\sU + \Vert x - x_k\Vert_\sX  \\ &< \Vert u_n - u_j \Vert_\sU + \Vert x - x_j\Vert_\sX \text{ for all } j\in\{1,\ldots,n\}\setminus\{k\}\big\}.
\end{align*}
By construction, $M_k$ contains all points $x\in\X$, such that $(u_n,x)$ is closer to $(u_k,x_k)$ than to any other previous point we evaluated $\nabla_1 j$ at.
The full gradient is then approximated in a piecewise constant fashion by
\begin{align*}
    \nabla J(u_n) &= \int_\X \nabla_1 j(u_n,x)\mu(\mathrm{d}x) \\ &= \sum_{k=1}^n \int_{M_k} \nabla_1 j(u_n,x)\mu(\mathrm{d}x) \\ &\approx \sum_{k=1}^n\nabla_1 j(u_k,x_k) \mu(M_k).
\end{align*}
Therefore, we call $\alpha_k^{\text{ex}} = \mu(M_k)$ \textit{exact} integration weights, since they are based on an exact nearest neighbor approximation. These weights were first introduced in \cite{pflug_CSG} and offer the best approximation to the full gradient. However, the calculation of the exact integration weights requires full knowledge of the measure $\mu$ and is based on a Voronoi tessellation \cite{voronoi}, which is computationally expensive for high dimensions of $\U\times\X$. 

Note that, in the special case $\dim(\X)=1$, the calculation of the tessellation can be circumvented, regardless of $\dim(\U)$. Instead, the intersection points of the line $\{u_n\}\times\X$ and all faces of active Voronoi cells can be obtained directly by solving the equations appearing in the definition of $M_k$. This, however, still requires us to solve $\mathcal{O}(n^2)$ quadratic equations per CSG iteration. 
\begin{figure}
  \begin{minipage}[c]{0.5\textwidth}
    \includegraphics[width=\textwidth]{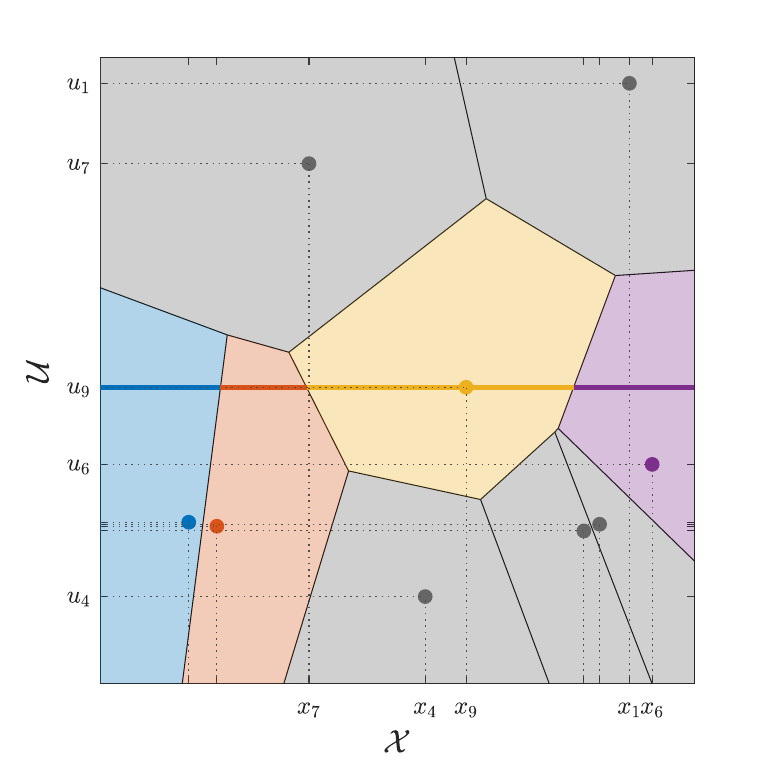}
  \end{minipage}\hfill
  \begin{minipage}[c]{0.45\textwidth}
    \caption{For the exact integration weights, $\alpha_i$ represents to the measure of the line segment that lies in the Voronoi cell around $(u_i,x_i)$. The grey cells correspond to samples with integration weight 0.}
    \label{fig:ExactWeights}
  \end{minipage}
\end{figure}

\subsection{Exact Hybrid Weights}
In some settings, the dimension of $\X$ can be very small compared to the dimension of $\U$. Hence, we might avoid computing a Voronoi tessellation in $\U\times\X$ by treating these spaces separately. For this, we introduce the sets
\begin{equation*}
    \widetilde{M}_i := \big\{ x\in\X\,:\, \Vert x-x_i\Vert_\sX < \Vert x-x_j\Vert_\sX \text{ for all }j\in\{1,\ldots,n\}\setminus\{i\}\big\}.
\end{equation*}
Denoting by $1_{M_k}$ the indicator function of $M_k$, we define the \textit{exact hybrid} weights as
\begin{equation*}
    \alpha_k^{\text{eh}} = \sum_{i=1}^n 1_{M_k}(x_i)\mu\big(\widetilde{M}_i\big).
\end{equation*}
Notice that the sets $M_k$ still appear in the definition of the exact hybrid weights, but do not need to be calculated explicitly. Instead, we only have to find the nearest sample point to $(u_n,x_i)_{i=1,\ldots,n}$, which can be done efficiently even in large dimensions. We do, however, still require knowledge of $\mu$. The idea of the exact hybrid weights is captured in \Cref{fig:ExHybridWeights}.
\begin{figure}
  \begin{minipage}[c]{0.5\textwidth}
    \includegraphics[width=\textwidth]{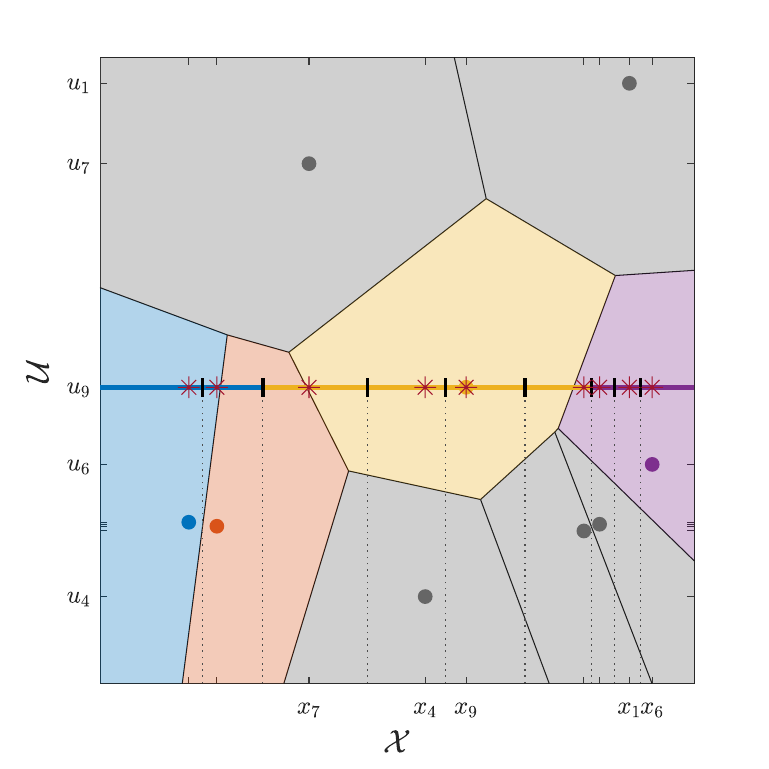}
  \end{minipage}\hfill
  \begin{minipage}[c]{0.45\textwidth}
    \caption{The red stars on the bold line represent the points $(u_n,x_i)$. For each cell, we determine the number of these points inside the cell and combine their corresponding measures $\mu\big(\widetilde{M}_i\big)$, indicated by the colored line segments. Since $\dim(\X) = 1$ in the shown example, the sets $\widetilde{M}_i$ are simply intervals, where the endpoints are given by the midpoints of two neighboring $x_i$.}
    \label{fig:ExHybridWeights}
  \end{minipage}
\end{figure}

\subsection{Inexact Hybrid weights}
To calculate the integration weights in the case that the measure $\mu$ is unknown, we may approximate $\mu\big(\widetilde{M}_i\big)$ empirically. All that is required for this approach is additional samples of the random variable $X$. To be precise, we draw enough samples such that in iteration $n$, we have in total $f(n)$ samples of $X$, where $f:\N\to\N$ is a function which is strictly increasing and satisfies $f(n)\ge n$ for all $n\in\N$. It is important to note, that we still evaluate $\nabla_1 j(u_n,\cdot)$ at only one of these points, which we denote by $x_{j_n}$. Thus, exchanging $\mu\big(\widetilde{M}_i\big)$ by its empirical approximation yields the \textit{inexact hybrid} weights
\begin{equation*}
    \alpha_k^{\text{ih}} = \frac{1}{f(n)} \sum_{i=1}^n 1_{M_k}(x_{j_i}) \sum_{m=1}^{f(n)} 1_{\widetilde{M}_{j_i}}(x_m).
\end{equation*}
How fast $f$ grows determines not only the quality of the approximation, but also the computational complexity of the weight calculation. Based on the choice of $f$, the inexact hybrid weights interpolate between the exact hybrid weights and the empirical weights, which will be introduced below. In \Cref{fig:InexactInterpol}, this behavior is shown for functions of the form $f(n) = \lfloor n^\beta \rfloor$ with $\beta\ge1$. \Cref{fig:InexHybridWeights} illustrates the general concept of the inexact hybrid weights.
\begin{figure}[h]
    \centering
    \begin{minipage}[t][][c]{0.48\textwidth}
        \centering
        \includegraphics[width = \textwidth]{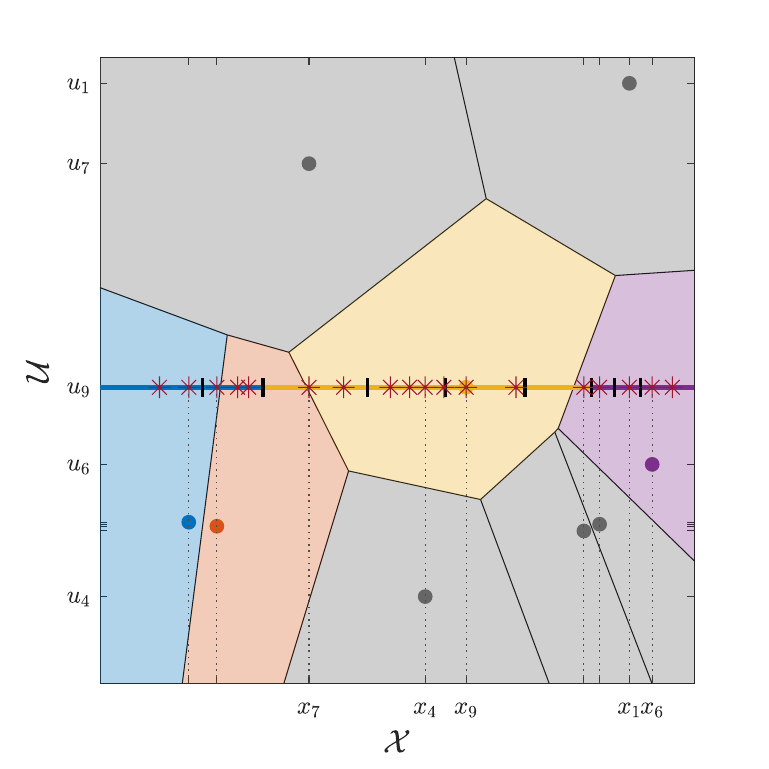}%
        \caption{The red stars on the bold line represent the points $(u_n,x_i)_{i=1,\ldots,f(n)}$. Similar to the exact hybrid weights, for each cell of the Voronoin diagramm, we count the points $(u_n,x_{j_i})_{i=1,\ldots,n}$ inside the cell (marked by dashed lines). Instead of assigning them the weights $\mu\big(\widetilde{M}_{j_i}\big)$, we weight them by an empirical approximation to this quantity, i.e., by the percentage of random samples $(x_i)_{i=1,\ldots,f(n)}$ that lie in $\widetilde{M}_{j_i}$. }
        \label{fig:InexHybridWeights}
    \end{minipage}\hfill
    \begin{minipage}[t][][c]{0.48\textwidth}
        \centering
        \includegraphics[width = \textwidth]{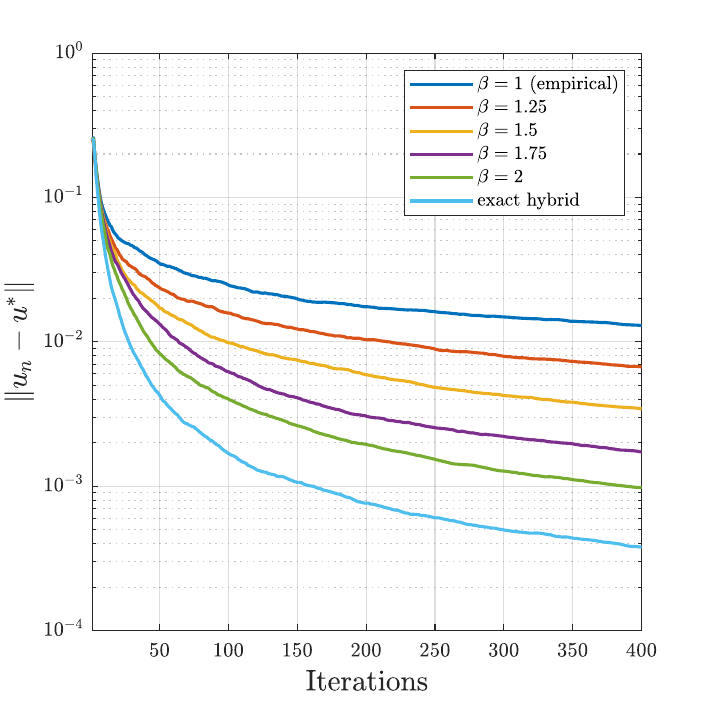}%
        \caption{Median of the results for problem \eqref{eq:ProblemConstSteps}, initialized with 1000 random starting points, for empirical weights, exact hybrid weights and several inexact hybrid weights. The function $f$ appearing in the inexact hybrid weights was chosen as $f(n)=\lfloor n^\beta \rfloor$.}
        \label{fig:InexactInterpol}
    \end{minipage}
\end{figure}
\subsection{Empirical Weights}
The \textit{empirical} weights offer a computationally cheap alternative to the methods mentioned above. Their calculation does not require any knowledge of the measure $\mu$. For the empirical weights, the quantity $\mu(M_k)$ is directly approximated by its empirical counterpart, i.e.,
\begin{equation*}
    \alpha_k^{\text{e}} = \frac{1}{n}\sum_{i=1}^n 1_{M_k}(x_i).
\end{equation*}
The corresponding picture is shown in \Cref{fig:EmpiricalWeights}. By iteratively storing the distances $\Vert x_i-x_j\Vert$, $i,j<n$, the empirical weights can be calculated very efficiently. 
\begin{figure}
  \begin{minipage}[c]{0.5\textwidth}
    \includegraphics[width=\textwidth]{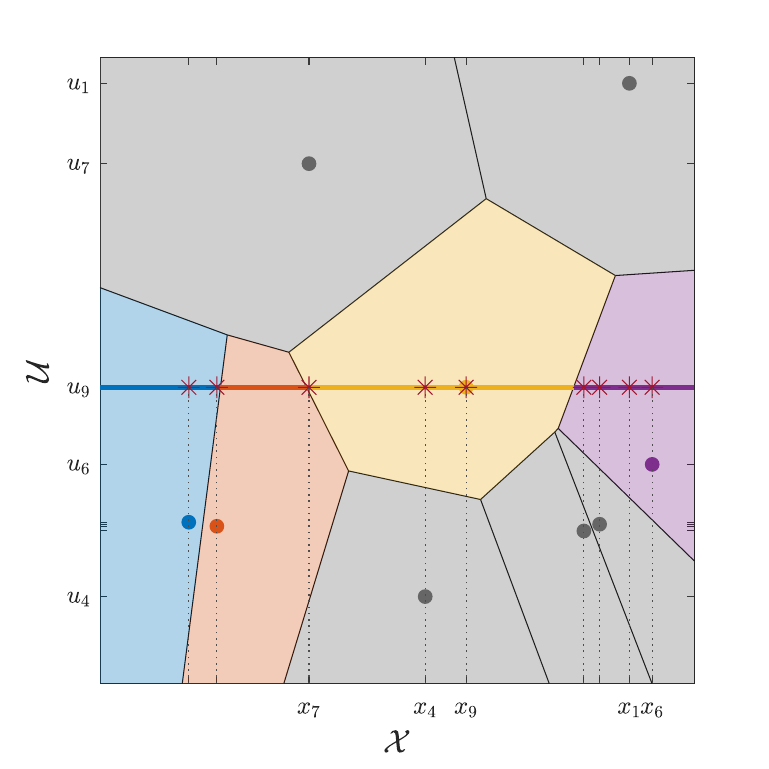}
  \end{minipage}\hfill
  \begin{minipage}[c]{0.45\textwidth}
    \caption{The empirical weights assign each line segment to its empirical measure, i.e., the fraction of points $(u_n,x_i)$ (red stars) which lie inside the corresponding cell.}
    \label{fig:EmpiricalWeights}
  \end{minipage}
\end{figure}

\subsection{Metric on $\U\times\X$}\label{rem:MetricWeights}
As mentioned after \Cref{assum:Weights}, the convergence results proven in this contribution are independent of the specific inner norms on $\U$ and $\X$, denoted by $\Vert\cdot\Vert_\sU$ and $\Vert\cdot\Vert_\sX$. Furthermore, they also hold if we substitute the outer norm on $\U\times\X$, appearing in \eqref{eq:khochn}, e.g., by a generalized $L^1$-outer norm, i.e., 
\begin{equation*}
    \big\Vert (u,x)\big\Vert_{\U\times\X} = c_1\Vert u\Vert_\sU + c_2\Vert x\Vert_\sX,
\end{equation*}
with $c_1,c_2>0$. While this does not seem particularly helpful at first glance, it in fact allows to drastically change the practical performance of CSG. By altering the ratio $\xi=\tfrac{c_2}{c_1}$, the CSG gradient approximation tends to consider fewer ($\xi<1$) or more ($\xi>1$) old samples in the linear combination. The effect such a choice can have in practice is visible in the numerical analysis of CSG in \cite{CSGPart2}.

To be precise, choosing $\xi\ll1$ results in the nearest neighbor being predominantly determined by the distance in design. In the extreme case, this means that CSG initially behaves more like a traditional SG algorithm. Analogously, $\xi\gg1$ will initially yield a gradient approximation, in which all old samples are used, even if they differ greatly in design (for discrete measure $\mu$, this corresponds to SAG). The corresponding Voronoi diagrams are shown in \Cref{fig:ChooseMetricL}.

For the sake of consistency, all numerical examples will use the euclidean norm $\Vert\cdot\Vert_{_2}$ as inner norms and we fix $\xi=1$.
\begin{figure}
    \centering
    \begin{minipage}[t][][c]{0.48\textwidth}
        \centering
        \includegraphics[width = \textwidth]{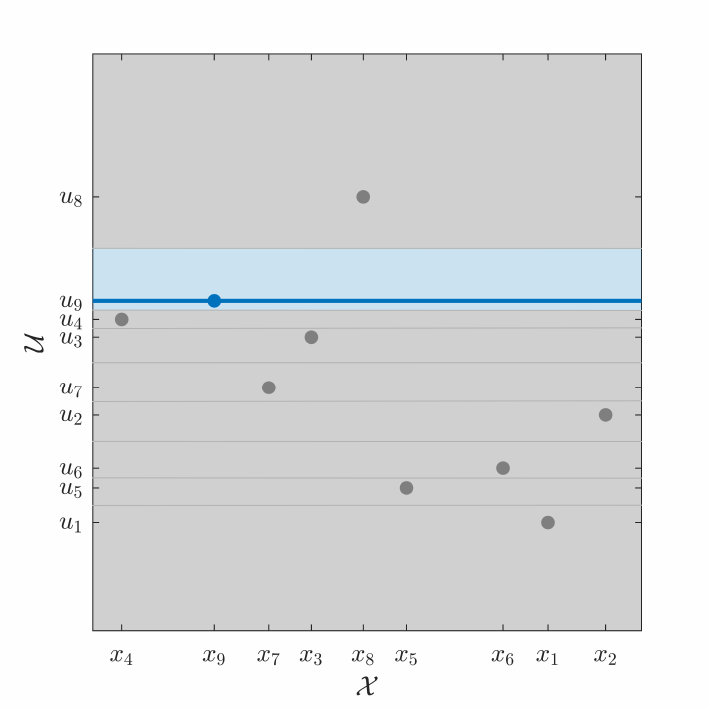}%
    \end{minipage}\hfill%
    \begin{minipage}[t][][c]{0.48\textwidth}
        \centering
        \includegraphics[width = \textwidth]{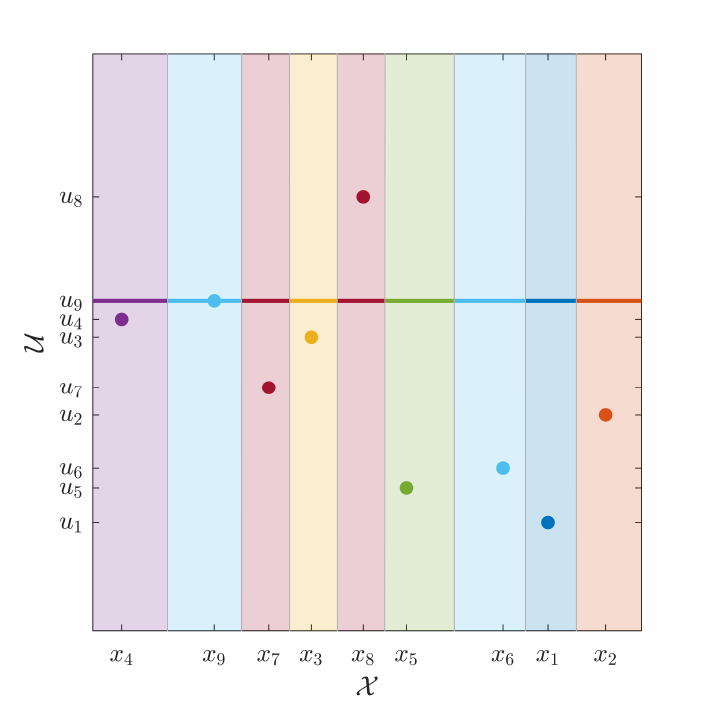}%
    \end{minipage}
    \caption{Voronoi diagrams for the integration weight calculation, given the sample points $(u_i,x_i)_{i=1,\ldots,9}$. Different ratios $\xi$ in the norm on $\U\times\X$ lead to SG-like behavior (left, $\xi = 0.01$) or averaging over all samples (right, $\xi=100$).}
    \label{fig:ChooseMetricL}
\end{figure}
\subsection{Properties of the Integration Weights}\label{sec:AppendixWeightsMeasures}
The integration weights gained by each of the methods mentioned above can be regarded as a special probability measure $\mu_n$ on the space $\X$, which we use to approximate the full integral. 
For example, the empirical weights correspond to the empirical measure (see, e.g., \cite{dudley1978central})
\begin{equation*}
    \mu_n^{\text{e}} := \frac{1}{n}\sum_{i=1}^n \delta_{x_i},
\end{equation*}
in the sense that
\begin{align*}
    \hat{G}_n^{\text{e}} &= \sum_{i=1}^n \alpha_i^{\text{e}}\nabla_1 j(u_i,x_i) = \sum_{i=1}^n \frac{1}{n}\sum_{m=1}^n 1_{M_i}(x_m) \nabla_1 j(u_i,x_i) \\
    &= \sum_{i=1}^n \frac{1}{n}\sum_{m=1}^n \delta_{k^n(x_m)}(i)\nabla_1 j(u_i,x_i) = \sum_{i=1}^n \int_\X \delta_{k^n(x)}(i)\nabla_1 j(u_i,x_i)\mu_n^{\text{e}}(\mathrm{d}x),
\end{align*}
where
\begin{equation*}
    k^n(x) := \argmin_{k = 1,\ldots,n} \big( \Vert u_n - u_k\Vert_\sU + \Vert x - x_k\Vert_\sX\big),
\end{equation*}
as in \Cref{assum:Weights}.
It was shown in \cite[Theorem 3]{EmpiricalDistrPaper}, that $\mu_n^{\text{e}}\Rightarrow \mu$ for $n\to\infty$, where $\Rightarrow$ denotes the weak convergence of measures (or equivalently, the weak-* convergence in dual space theory, see, e.g., \cite[Section 7.3]{folland2009guide}):
\begin{equation*}
    \nu_n\Rightarrow\nu \quad \text{iff} \quad \int_\X f(x)\nu_n(\mathrm{d}x) \to \int_\X f(x)\nu(\mathrm{d}x)\quad \text{for all }f\in\mathcal{C}_b(\X,\R).
\end{equation*}
Likewise, the measures associated with the other integration weights
\begin{equation*}
    \mu_n^{\text{ex}} = \mu,\quad \mu_n^{\text{eh}} = \sum_{i=1}^n \delta_{x_i}\mu\big(\widetilde{M}_i\big) \quad \text{and} \quad \mu_n^{\text{ih}} = \sum_{i=1}^n \delta_{x_{j_i}} \mu_{f(n)}^{\text{e}}\big(\widetilde{M}_{j_i}\big)
\end{equation*}
satisfy $\mu_n^{\text{ex}}\Rightarrow \mu$, $\mu_n^{\text{eh}}\Rightarrow\mu$ and $\mu_n^{\text{ih}}\Rightarrow\mu$ as well. This property plays an important role in the proof of \Cref{lem:GradientError}, as it ensures a vanishing Wasserstein distance of $\mu_n$ and $\mu$ for $n\rightarrow \infty$, i.e. $d_W (\mu_n,\mu)\to 0$ for all of the presented integration weights, see \cite[Theorem 6]{gibbs2002choosing}.

\subsection{Batches, Patches and Parallelization}
All methods for obtaining the integration weights presented above heavily relied on evaluating the gradient $\nabla_1 j$ only at a single sample point $(u_n,x_n)$ per iteration. In other words, \Cref{alg:GeneralCSG} has a natural batch size of 1. While we certainly do not want to increase this number too much, stochastic mini-batch algorithms outperform classic SG in some instances, especially if the evaluation of the gradient samples $\nabla_1 j\big(u_n,x_n^{(i)}\big)_{i=1,\ldots,N}$ can be done in parallel.

Increasing the batch size $N$ in CSG leaves the question of how the integration weights should be obtained. We could, of course, simply collect all evaluation points and calculate the weights as usual. This, however, may significantly increase the computational cost, effectively scaling it by $N$. In many instances, there is a much more elegant solution to this problem, which lets us include mini-batches without any additional weight calculation cost.

Assume, for simplicity, that $\X$ is an open interval and that $\mu$ corresponds to a uniform distribution, i.e., there exist $a<b\in\R$ such that
\begin{equation*}
    J(u) = \frac{1}{b-a}\int_a^b j(u,x)\mathrm{d}x.
\end{equation*}
Given $N\in\N$, we obtain
\begin{align*}
    J(u) &= \frac{1}{b-a}\int_a^b j(u,x)\mathrm{d}x = \frac{1}{b-a}\sum_{i=1}^N \int_{a+(i-1)\tfrac{b-a}{N}}^{a + i\tfrac{b-a}{N}} j(u,x)\mathrm{d}x \\
    &= \frac{1}{b-a}\int_{a}^{a+\tfrac{b-a}{N}} \sum_{i=1}^N j\left(u,x+(i-1)\tfrac{b-a}{N}\right)\mathrm{d}x =: \frac{1}{b-a}\int_{a}^{a+\tfrac{b-a}{N}} \tilde{j}(u,x)\mathrm{d}x
\end{align*}
Thus, we can include mini-batches of size $N\in\N$ into CSG by performing the following steps in each iteration:
\begin{itemize}
    \item[1.)] Draw random sample $x_n\in\left(a,a+\tfrac{b-a}{N}\right)$.
    \item[2.)] Evaluate $\nabla_1 j(u_n,\cdot)$ at each
    \begin{equation*}
        x_n^{(i)} = x_n + (i-1)\tfrac{b-a}{N},\quad i=1,\ldots,N.
    \end{equation*}
    \item[3.)] Compute 
    \begin{equation*}
        \nabla_1 \tilde{j}(u_n,x_n) = \sum_{i=1}^N \nabla_1 j\big(u_n,x_n^{(i)}\big).
    \end{equation*}
    \item[4.)] Calculate integration weights $(\alpha_k)_{k=1,\ldots,n}$ as usual for $(u_k,x_k)_{k=1,\ldots,n}$ and set
    \begin{equation*}
        \hat{G}_n = \sum_{k=1}^n \alpha_k \nabla_1 \tilde{j}(u_k,x_k).
    \end{equation*}
\end{itemize}
It is straightforward to apply this idea to higher dimensional rectangular cuboids. Furthermore, the process of subdividing $\X$ into smaller patches and drawing the samples in only one of these regions $\X_1\subset\X$ can be generalized to more complex settings as well. However, it is necessary that translating the sample $x_n$ into the other patches preserves the underlying probability distribution $X\sim\mu$, e.g., if $\mu$ is induced by a Lipschitz continuous density function and the different patches are obtained by reflecting, translating, rotating and scaling $\X_1$. A conceptual example is shown in \Cref{fig:batchidea}.

The effect of introducing mini-batches into CSG is tested for the optimization problem
\begin{equation}
    \min_{u\in[-5,5]^2}\quad \frac{1}{2}\int_{\X}\Vert u-x\Vert_2^2 \mathrm{d}x, \label{eq:BatchProblem}
\end{equation}
by dividing $\X=\left(-\tfrac{1}{2},\tfrac{1}{2}\right)^2$ into small squares of sidelength $\tfrac{1}{N}$, achieving a batch size of $N^2$. The results of 500 optimization runs with random starting points are given in \Cref{fig:batchrates2}. Although increasing the batch size does, as expected, not influence the rate of convergence, it still improves the overall performance and should definitely be considered for complex optimization problems, especially if the gradient sampling can be performed in parallel.
\begin{figure}[htp]
    \centering
    \begin{minipage}[b][][c]{0.48\textwidth}
        \centering
        \resizebox{\textwidth}{!}{
        \begin{tikzpicture}[scale = 1]
        \fill[black!20] (3,0) -- (5,0) -- (5,3) -- (4,3) -- (4,2) -- (3,2) -- (3,0);
        \draw[very thick] (0,0) -- (2,0) -- (2,1) -- (3,1) -- (3,0) -- (5,0) -- (5,1) -- (6,1) -- (6,0) -- (8,0) -- (8,3) -- (7,3) -- (7,4) -- (5,4) -- (5,5) -- (3,5) -- (3,4) -- (2,4) -- (2,5) -- (1,5) -- (1,2) -- (0,2) -- (0,0);
        \draw[dashed,color = blue] (3,4) -- (3,1);
        \draw[dashed,color = blue] (5,4) -- (5,1);
        \draw[dashed,color = blue] (4,2) -- (4,3);
        \draw[dashed,color = blue] (1,2) -- (4,2);
        \draw[dashed,color = blue] (5,2) -- (8,2);
        \draw[dashed,color = blue] (4,3) -- (5,3);
        \filldraw[red] (3.5,1) circle (3pt) node[anchor = north west]{\Large $x_n$};
        \filldraw[red!40] (1,0.5) circle (3pt);
        \filldraw[red!40] (2.5,3) circle (3pt);
        \filldraw[red!40] (4.5,4) circle (3pt);
        \filldraw[red!40] (6,3.5) circle (3pt);
        \filldraw[red!40] (7,0.5) circle (3pt);
        \end{tikzpicture}
        }
    \end{minipage}\hfill
    \begin{minipage}[b][][c]{0.48\textwidth}
        \centering
        \includegraphics[width = \textwidth]{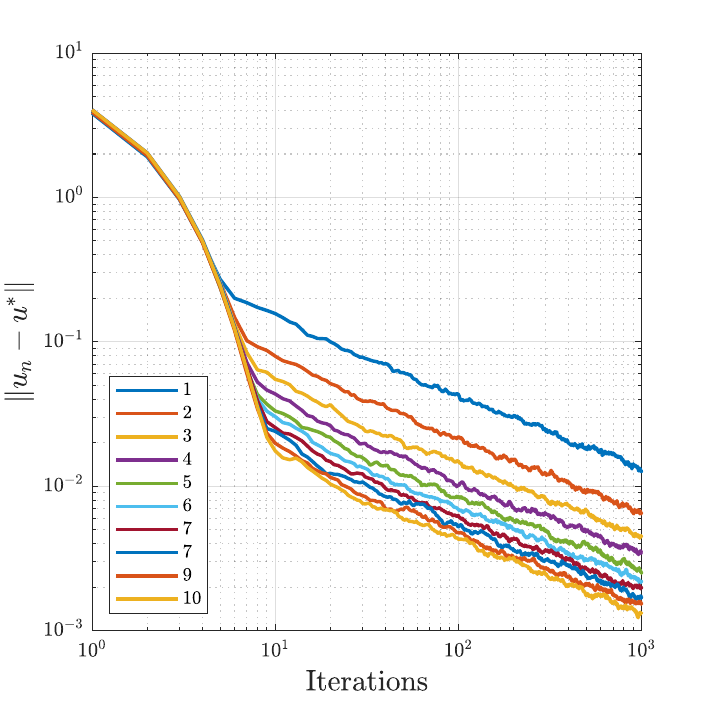}%
    \end{minipage}
    \par
    \begin{minipage}[t][][c]{0.48\textwidth}
        \caption{To obtain a mini-batch of six samples points, the domain $\X$ is subdivided into six patches, indicated by the blue lines. In each iteration, a sample point $x_n$ is drawn in patch $\X_1$ (grey) and afterwards translated into all other patches.}
        \label{fig:batchidea}
    \end{minipage}\hfill
    \begin{minipage}[t][][c]{0.48\textwidth}
        \caption{Median distance of the CSG iterates to the optimal solution $u^\ast = 0\in\R^2$ of \eqref{eq:BatchProblem} in each iteration. The batch size used is equal to $N^2$, where different values of $N$ are indicated by the different colors. For all runs, a constant step size of $\tau = \tfrac{1}{2}$ was chosen.}
        \label{fig:batchrates2}
    \end{minipage}
\end{figure}
\section{Auxiliary Results}\label{sec:Auxiliary}
From now on, unless explicitly stated otherwise, we will always assume \Cref{assum:SampleSequence,assum:Nummer2,assum:RegularitySmallj,assum:Weights} to be satisfied.
We will later show that the CSG method converges to stationary points, which we define next.
\begin{definition}[Stationary Points]\label{def:StationaryPoints} Let $J\in C^1(\U)$ be given. We say $u^\ast\in\U$ is a stationary point of $J$, if there exists $t>0$ such that

\begin{equation*}
    \P_\U\left( u^\ast - t\nabla J(u^\ast)\right) = u^\ast.
\end{equation*}
Furthermore, we denote by
\begin{equation*}
    \mathcal{S}(J):= \left\{ u\in\U\,:\, \P_\U\left(u-t\nabla J(u)\right) = u \text{ for some }t>0\right\}
\end{equation*}
the set of all stationary points of $J$.
\end{definition}
Gradient descent methods for $L$-smooth objective functions have thoroughly been studied in the past (e.g. \cite{ProjGradientGoldstein,ProjGradientRates}). The key ingredients for obtaining convergence results with constant step sizes are the descent lemma and the characteristic property of the projection operator, which we state in the following.
\begin{lemma}[Descent Lemma]\label{lem:descent lemma}
If $J:\U\to\R$ is $L$-smooth, then it holds
\[
J(u_1)\le J(u_2) +\nabla J(u_2)^\top (u_1-u_2) +\frac{L}{2}\Vert u_1-u_2\Vert_\sU^2\quad \forall u_1,u_2\in\U.
\]
\end{lemma}

\begin{lemma}[Characteristic Property of Projection]\label{lem: char property}
For a projected step in direction $\hat{G}_n$, i.e., $u_{n+1}=\P_\U(u_n-\tau_n\hat{G}_n)$, the following holds:
\[
\hat{G}_n^\top(u_n-u_{n+1}) \ge \frac{\Vert u_n-u_{n+1}\Vert_\sU^2}{\tau_n}.
\]
\end{lemma}
\begin{proof}
\Cref{lem:descent lemma} corresponds to \cite[Lemma 5.7]{DescentUndProjection}. \Cref{lem: char property} is a direct consequence of \cite[Theorem 6.41]{DescentUndProjection}.
\end{proof}

Before we move on to results that are specific for the CSG method, we state a general convergence result, which will be helpful in the later proofs.

\begin{lemma}[Finitely Many Accumulation Points]\label{lem:Convergence}
Let $(u_n)_{n\in\N}\subset\R^d$ be a bounded sequence. Suppose that $(u_n)_{n\in\N}$ has only finitely many accumulation points and it holds $\Vert u_{n+1}-u_n\Vert\to 0$. Then $(u_n)_{n\in\N}$ is convergent. 
\end{lemma}
\begin{proof}
Let $\left\{\bar{u}_1,\ldots,\bar{u}_K\right\}$ be the accumulation points of $(u_n)_{n\in\N}$ and define
\[
\delta_0 := \min_{\substack{i,j\in\{1,\ldots,K\} \\ i\neq j}} \Vert \bar{u}_i-\bar{u}_j\Vert,
\]
i.e, the minimal distance between two accumulation points of $(u_n)_{n\in\N}$. The accumulation point closest to $u_n$ is defined as:
\[
\bar{u}(n) := \argmin_{u\in\{\bar{u}_1,\ldots,\bar{u}_K\}} \Vert u_n-u\Vert.
\]

Next up, we show that there exists $N\in\N$ such that for all $n\ge N$ it holds $\Vert u_n-\bar{u}(n)\Vert < \tfrac{\delta_0}{4}$. We prove this by contradiction.

Thus, we assume there exist infinitely many $n\in\N$ such that $\Vert u_n-\bar{u}(n)\Vert \ge \tfrac{\delta_0}{4}$. This subsequence is again bounded and therefore must have an accumulation point. By construction, this accumulation point is no accumulation point of $(u_n)_{n\in\N}$, which is a contradiction.

Now, let $N_1\in\N$ be large enough such that $\Vert u_n-\bar{u}(n)\Vert<\tfrac{\delta_0}{4}$ for all $n\ge N_1$. By our assumptions, there also exists $N_2\in\N$ with $\Vert u_{n+1}-u_n\Vert<\tfrac{\delta_0}{4}$ for all $n\ge N_2$. Define $N:=\max\{N_1,N_2\}$.

Let $n \geq N$ and assume for contradiction that $\bar{u}(n) \neq \bar{u}(n+1)$. We obtain
\begin{equation*}
\text{dist}\big( \mathcal{B}_{\delta_0/4}(\bar{u}(n))\, , \,  \mathcal{B}_{\delta_0/4}(\bar{u}(n+1))\big) \ge \tfrac{\delta_0}{2} > \tfrac{\delta_0}{4} > \Vert u_{n}-u_{n+1}\Vert \quad \text{for all } n\ge N,
\end{equation*}
where $\dist(A,B) := \inf_{x\in A,y\in B}\Vert x-y\Vert$ for $A,B\subset\R^d$. This is a contradiction to  $\Vert u_{n+1}-u_n\Vert<\tfrac{\delta_0}{4}$ for all $n \geq N$.

We thus conclude that $u_n\in\mathcal{B}_{\delta_0/4}\big(\bar{u}(n)\big)$ implies $u_{n+1}\in\mathcal{B}_{\delta_0/4}\big(\bar{u}(n)\big)$ as well. Since $\bar{u}(n)$ is the only accumulation point on $\mathcal{B}_{\delta_0/4}\big(\bar{u}(n)\big)$, it follows that $u_n\to\bar{u}(N)$.

\end{proof}


\subsection{Results for CSG Approximations}
From now on, let $(u_n)_{n\in\N}$ denote the sequence of iterates generated by \Cref{alg:GeneralCSG}. In this section, we want to show that the CSG approximations $\hat{J}_n$ and $\hat{G}_n$ in the course of iterations approach the values of $J(u_n)$ and $\nabla J(u_n)$, respectively. This is a key result for the convergence theorems stated in \Cref{sec:ConvergenceConstant} and \Cref{sec:ConvergenceArmijo}.
\begin{lemma}[Density Result in $\X$]\label{lem:ErstesResultat}
Let $(x_n)_{n\in\N}$ be the random sequence appearing in \Cref{alg:GeneralCSG}. For all $\e>0$  there exists $N\in\N$ such that
\begin{equation*}
    \min_{n\in\{1,\ldots,N\}} \Vert x_n -x\Vert_\sX <\e \quad\text{for all }x\in\supp(\mu).
\end{equation*}
\end{lemma}
\begin{proof}
Utilizing the compactness of $\supp(\mu)\subset\R^\dran$, there exists $T\in\N$ such that $\left( \mathcal{B}_{\e/2}(m_i)\right)_{i=1,\ldots,T}$ is an open cover of $\supp(\mu)$ consisting of balls with radius $\frac{\e}{2}$ centered at points $m_i\in\supp(\mu)$. Thus, for each $x\in\supp(\mu)$ we can find $i_x\in\{1,\ldots,T\}$ with $x\in\mathcal{B}_{\e/2}(m_{i_x})$. Hence, by \Cref{rem:SequenceDense}, for each $i=1,\ldots,T$, there exists $n_i\in\N$ satisfying
\begin{equation*}
    \Vert x_{n_i}-m_i\Vert_\sX < \tfrac{\e}{2}.
\end{equation*}
Defining 
\begin{equation*}
    N := \max_{i\in\{1,\ldots,T\}} n_i < \infty,
\end{equation*}
for all $x\in\supp(\mu)$ we have
\begin{equation*}
    \begin{aligned}
    \min_{n\in\{1,\ldots,N\}} \Vert x - x_n\Vert_\sX &\le \min_{n\in\{1,\ldots,N\}}\big( \Vert x-m_{i_x}\Vert_\sX + \Vert m_{i_x}-x_n\Vert_\sX\big) \\
    &< \tfrac{\e}{2}+\min_{n\in\{1,\ldots,N\}}\Vert m_{i_x}-x_n\Vert_\sX < \tfrac{\e}{2}+\tfrac{\e}{2} = \e.
    \end{aligned}
\end{equation*}
\end{proof}

\begin{lemma}[Density Result in $\U \times \X$]\label{lem:ZweitesResultat}
Let $(u_n)_{n\in\N}$, $(x_n)_{n\in\N}$ be the sequences of optimization variables and sample sequence appearing in \Cref{alg:GeneralCSG}. For all $\e>0$ there exists $N\in\N$ such that
\begin{equation*}
    Z_n(x) := \min_{k\in\{1,\ldots,n\}}\big(\Vert u_n-u_k\Vert_\sU + \Vert x-x_k\Vert_\sX\big) < \e 
\end{equation*}
for all $n>N$ and all $x\in\supp(\mu)$.
\end{lemma}
\begin{proof}
Since $\U$ is compact, we can find a finite cover $\left( \mathcal{B}_{\e/4}(m_i)\right)_{i=1,\ldots,T}$ of $\U$ consisting of $T\in\N$ balls with radius $\frac{\e}{4}$ centered at points $m_i\in\U$. Define $I\subset\{1,\ldots,T\}$ as
\begin{equation*}
    I := \big\{ i\in\{1,\ldots,T\}\, :\, u_n\in\mathcal{B}_{\e/4}(m_i)\text{ for infinitely many }n\in\N\big\}.
\end{equation*}
By our definition of $I$, for each $i\in\{1,\ldots,T\}\setminus I$ there exists $\tilde{N}_i\in\N$ such that ${u_n\not\in\mathcal{B}_{\e/4}(m_i)}$ for all $n>\tilde{N}_i$. Setting 
\begin{equation*}
    N_1 := \max_{i\in\{1,\ldots,T\}\setminus I} \tilde{N}_i,
\end{equation*}
it follows that 
\begin{equation} \label{eq:NotInBall}
    u_n\not\in\mathcal{B}_{\e/4}(m_i) \text{ for each $n>N_1$ and all $i\in\{1,\ldots,T\}\setminus I$.}
\end{equation}
For $i\in I$ let $\big( u_{n^{(i)}_t}\big)_{t\in\N}$ be the subsequence consisting of all elements of $(u_n)_{n\in\N}$ that lie in $\mathcal{B}_{\e/4}(m_i)$. Observe that $\big(x_{n_t^{(i)}}\big)_{t\in\N}$ is independent and identically distributed according to $\mu$, since for all $A\subseteq \supp(\mu)$ and each $i\in I$, it holds
\begin{equation*}
    \mathbb{P}\left( x_{n}\in A\, \middle\vert\, u_{n}\in\mathcal{B}_{\e/4}(m_i)\right) = \mu(A)\quad\text{for all }n\in\N.
\end{equation*}
Thus, by \Cref{rem:SequenceDense}, $\big(x_{n_t^{(i)}}\big)_{t\in\N}$ is dense in $\supp(\mu)$ with probability 1 for all $i\in I$.
By \Cref{lem:ErstesResultat}, we can find $K_i\in\N$ such that
\begin{equation}\label{eq:AbstandKlein}
    \min_{t\in\{1,\ldots,K_i\}} \big\Vert x-x_{n^{(i)}_t}\big\Vert_\sX < \tfrac{\e}{2}\quad\text{ for all }x\in\supp(\mu).
\end{equation}
Define 
\begin{equation*}
    N_2 := \max_{i\in I}\max_{t\in\{1,\ldots,K_i\}} n^{(i)}_t
\end{equation*}
as well as $N := \max (N_1,N_2)$. Notice that this definition of $N$ implies for all $i\in I$ and all $n>N$
\begin{equation*}
    \left\{ n^{(i)}_t \, :\, t\in \{ 1,\ldots, K_i\}\right\} \subseteq \{ 1,\ldots,n\}.
\end{equation*}
By \eqref{eq:NotInBall}, for all $n>N$ there exists $i\in I$ such that
\begin{equation*}
    \Vert u_n - m_i\Vert_\sU < \tfrac{\e}{4}. 
\end{equation*}
Now, given $x\in\supp(\mu)$ and $n> N$, we choose $j\in I$ such that $u_n\in \mathcal{B}_{\varepsilon/4}(m_j)$. Thus, it holds
\begin{equation*}
    \begin{aligned}
    Z_n(x) &= \min_{k\in\{1,\ldots,n\}}\big(\Vert u_n-u_k\Vert_\sU + \Vert x-x_k\Vert_\sX\big) \\
        &\le \min_{t\in\{1,\ldots,K_j\}} \Big( \big\Vert u_n - u_{n^{(j)}_t}\big\Vert_\sU + \big\Vert x- x_{n^{(j)}_t}\big\Vert_\sX\Big) \\
        &\le \min_{t\in\{1,\ldots,K_j\}} \Big( \Vert u_n - m_j\Vert_\sU + \big\Vert m_j - u_{n^{(j)}_t}\big\Vert_\sU + \big\Vert x- x_{n^{(j)}_t}\big\Vert_\sX\Big) \\
        &< \tfrac{\e}{4} + \tfrac{\e}{4} + \tfrac{\e}{2} = \e,
    \end{aligned}
\end{equation*}
where we used \eqref{eq:AbstandKlein} and $u_n,u_{n^{(j)}_t}\in\mathcal{B}_{\e/4}(m_j)$ in the last line.
\end{proof}
\begin{lemma}[Approximation Results for $J$ and $\nabla J$]\label{lem:GradientError}
The approximation errors for $\nabla J$ and $J$ vanish in the course of the iterations, i.e.,
\begin{equation*}
    \left\Vert \hat{G}_n - \nabla J(u_n)\right\Vert \to 0 \quad\text{and}\quad \left\Vert \hat{J}_n - J(u_n)\right\Vert \to 0 \qquad \text{for } n\to\infty.
\end{equation*}

\end{lemma}
\begin{proof}
Denote by $\nu_n$ the measure corresponding to the integration weights according to \eqref{eq:assumMeasure1}. For $x\in\supp(\mu)$, we define the closest index in the current iteration $k^n(x)$ as
\begin{equation*}
    k^n(x) := \argmin_{k = 1,\ldots,n} \big( \Vert u_n - u_k\Vert_\sU + \Vert x - x_k\Vert_\sX\big),
\end{equation*}
i.e., we have
\begin{equation*}
    \Vert u_n - u_{k^n(x)}\Vert_\sU + \Vert x - x_{k^n(x)}\Vert_\sX = Z_n(x).
\end{equation*}
Now, it holds
\begin{align*}
    \big\Vert \hat{G}_n - &\nabla J(u_n)\big\Vert \\
    &= \left\Vert \sum_{i=1}^n \int_\X \delta_{k^n(x)}(i) \nabla_1 j(u_i,x_i)\nu_n (\mathrm{d}x) -\int_\X \nabla_1 j(u_n,x)\mu(\mathrm{d}x)\right\Vert \\
    &\le \left\Vert \int_\X \Big(\sum_{i=1}^n \delta_{k^n(x)}(i)\nabla_1 j(u_i,x_i) - \nabla_1 j(u_n,x)\Big)\nu_n(\mathrm{d}x)\right\Vert \\
    &\quad + \left\Vert \int_\X \nabla_1 j(u_n,x)\nu_n(\mathrm{d}x) - \int_\X \nabla_1 j(u_n,x)\mu(\mathrm{d}x)\right\Vert \\
    &\le L_j \int_\X Z_n(x)\nu_n(\mathrm{d}x) + \left\Vert \int_\X \nabla_1 j(u_n,x)\nu_n(\mathrm{d}x) - \int_\X \nabla_1 j(u_n,x)\mu(\mathrm{d}x)\right\Vert,
\end{align*}
where $L_j$ is the Lipschitz constant of $\nabla_1 j$ as defined in \Cref{assum:RegularitySmallj}. 

First, since $Z_n$ is uniformly (in $n$) Lipschitz continuous, we obtain
\begin{align*}
    \int_\X Z_n(x)\nu_n(\mathrm{d}x) &= \int_\X Z_n(x)\mu(\mathrm{d}x) + \int_\X Z_n(x)\nu_n(\mathrm{d}x) - \int_\X Z_n(x)\mu(\mathrm{d}x) \\
    &\le \int_\X Z_n(x)\mu(\mathrm{d}x) + L_Z\cdot d_W(\nu_n,\mu).
\end{align*}
Here, $L_Z$ corresponds to the Lipschitz constant of $Z_n$ and $d_W$ denotes the Wasserstein distance of the measure $\nu_n$ and $\mu$. By \Cref{assum:Nummer2}, $\X$ is bounded and by \Cref{assum:Weights}, we have $\nu_n\Rightarrow\mu$. Thus, \cite[Theorem 6]{gibbs2002choosing} yields $d_W(\nu_n,\mu)\to0$. Additionally, since $Z_n$ is bounded and converges pointwise to 0 (see \Cref{lem:ZweitesResultat}), we use Lebesgue's dominated convergence theorem and conclude
\begin{equation*}
    \int_\X Z_n(x)\mu(\mathrm{d}x) \to 0 \quad\text{for }n\to\infty.
\end{equation*}
For the second part, let $Q_n$ be an arbitrary coupling of $\nu_n$ and $\mu$, i.e., $Q_n(\cdot\times\X) = \nu_n$ and $Q_n(\X\times \cdot)=\mu$. Utilizing the Lipschitz continuity of $\nabla_1 j$ (\Cref{assum:RegularitySmallj}) once again, we obtain
\begin{align*}
    &\left\Vert \int_\X \nabla_1 j(u_n,x)\nu_n(\mathrm{d}x) - \int_\X \nabla_1 j(u_n,x)\mu(\mathrm{d}x)\right\Vert \\
    &\qquad \le \left\Vert \int_{\X\times\X} \big( \nabla_1 j(u_n,x) - \nabla_1 j(u_n,y)\big) Q_n(\mathrm{d}(x,y))\right\Vert \\
    &\qquad \le L_j \int_{\X\times\X} \Vert x-y\Vert_\sX Q_n(\mathrm{d}(x,y)).
\end{align*}
Denote the set of all couplings of $\nu_n$ and $\mu$ by $\mathbf{Q}$. Since $Q_n$ was arbitrary, it holds 
\begin{align*}
    &\left\Vert \int_\X \nabla_1 j(u_n,x)\nu_n(\mathrm{d}x) - \int_\X \nabla_1 j(u_n,x)\mu(\mathrm{d}x)\right\Vert \\
    &\qquad \le L_j\cdot \inf_{Q_n\in\mathbf{Q}} \int_{\X\times\X} \Vert x-y\Vert_\sX Q_n(\mathrm{d}(x,y)) \\
    &\qquad = L_j\cdot d_W(\nu_n,\mu) \to 0 \quad\text{for }n\to\infty,
\end{align*}
finishing the proof of $\left\Vert \hat{G}_n - \nabla J(u_n)\right\Vert \to 0$. The second part of the claim follows analogously. 

\end{proof}
As a final remark before starting the convergence analysis, we want to give further details on the class of problems that can be solved by the CSG algorithm.
\begin{remark}[Generalized Setting]\label{rem:GeneralizedSetting}
Suppose that, in addition to $\U,\X$ and $J$ as defined in the introduction, we are given a convex set $\mathcal{V}\subset\R^{d_1}$ for some $d_1\in\N$ and a continuously differentiable function $F:\mathcal{V}\times\R\to\R$. Now, if we consider the optimization problem
\begin{align}
    \min_{(u,v)\in\U\times\mathcal{V}}&\quad F(v,J(u)),\label{eq:FirstPertubedProblem}
\end{align}
the gradient of the objective function with respect to $(u,v)$ is given by
\begin{equation*}
    \nabla F(v,J(u)) = \begin{pmatrix} \nabla_1 F(v,J(u)) \\ \nabla J(u)\partial_2 F(v,J(u))\end{pmatrix}. 
\end{equation*}
It is a direct consequence of \Cref{lem:GradientError}, that 
\begin{equation*}
    \left\Vert \begin{pmatrix} \nabla_1 F(v_n,J(u_n)) \\ \hat{G}_n \partial_2 F(v_n,\hat{J}_n) \end{pmatrix} - \nabla F(v_n,J(u_n))\right\Vert \to 0\quad \text{for }n\to\infty.
\end{equation*}
Thus, we can use the CSG method to solve \eqref{eq:FirstPertubedProblem} and all our convergence results carry over to this setting, as long as the new objective function satisfies \Cref{assum:RegularitySmallj}.

Furthermore, let $\mathcal{Y}\subset\R^{d_2}$ for some $d_2\in\N$. Assume that we are given a probability measure $\nu$ such that the pair $(\mathcal{Y},\nu)$ satisfies the same assumptions we imposed on $(\X,\mu)$ and consider the optimization problem
\begin{equation}
    \min_{(u,v)\in\U\times\mathcal{V}}\quad \int_\mathcal{Y} \tilde{F}(v,J(u),y) \nu(\mathrm{d}y). \label{eq:SecondPertubedProblem}
\end{equation}
Again, the gradient of this objective function can be approximated by the CSG method, if $\tilde{F}:\mathcal{V}\times\R\times\mathcal{Y}\to\R$ is Lipschitz continuously differentiable. 

It is clear that we can continue to wrap around functions or expectation values in these fashions indefinitely. Therefore, we see that the scope of the CSG method is far larger than problems like \eqref{eq:Problemstellung} and includes many settings, which stochastic gradient descent methods can not handle, like nested expected values, tracking of expected values and many more.
\end{remark}
\subsection{Example for a Composite Objective Function}\label{subsec:ExampleGeneralizedSetting}
To study the performance of CSG in the generalized setting, we consider an optimization problem in which the objective function is not of the type \eqref{eq:Problemstellung}, but instead falls in the broader class of possible settings mentioned in \Cref{rem:GeneralizedSetting}. Thus, we introduce the sets $\U = [0,10]$, $\X = (-1,1)$ and $\mathcal{Y}=(-3,3)$ and define the optimization problem
\begin{equation}
    \min_{u\in\U}\quad \frac{1}{20}\int_{\mathcal{Y}}\Big( 2y+5\int_\X \cos\big(\tfrac{u-x}{\pi}\big)\mathrm{d}x\Big)^2\mathrm{d}y.\label{eq:SCGDComp}
\end{equation}
The optimal solution $u^\ast=\tfrac{\pi^2}{2}$ to \eqref{eq:SCGDComp} can be found analytically. The nonlinear fashion in which the inner integral over $\X$ enters the objective function prohibits us from using SG-type methods to solve \eqref{eq:SCGDComp}. There is, however, the possibility to use the stochastic compositional gradient descent method (SCGD), which was proposed in \cite{SCGDPaper} and is specifically designed for optimization problems of the form \eqref{eq:SCGDComp}. Each SCGD iteration consists of two main steps: The inner integral is approximated by samples using iterative updating with a slow-diminishing step size. This approximation is then used to carry out a stochastic gradient descent step with a fast-diminishing step size.

For numerical comparison, we choose 1000 random starting points in $[\tfrac{11}{2},\tfrac{19}{2}]$, i.e., the right half of $\U$. In our tests, the accelerated SCGD method (see \cite{SCGDPaper}) performed better than basic SCGD, mainly since the objective function of \eqref{eq:SCGDComp} is strongly convex in a neighborhood of $u^\ast$. Thus, we compare the results obtained by CSG to the aSCGD algorithm, for which we chose the optimal step sizes according to \cite[Theorem 7]{SCGDPaper}. For CSG, we chose a constant step size $\tau = \tfrac{1}{30}$, which represents a rough approximation to the inverse of the Lipschitz constant $L$. The results are given in \Cref{fig:SCGDCSG}.

Furthermore, we are interested in the number of steps each method has to perform such that the distance to $u^\ast$ lies (and stays) within a given tolerance. Thus, we also analyzed the number of steps after which the different methods obtain a result within a tolerance of $10^{-1}$ in at least 90\% of all runs. The results are shown in \Cref{fig:SCGDCSG2}.
\begin{figure}[htp]
    \centering
    \begin{minipage}[b][][c]{0.48\textwidth}
        \centering
        \includegraphics[width = \textwidth]{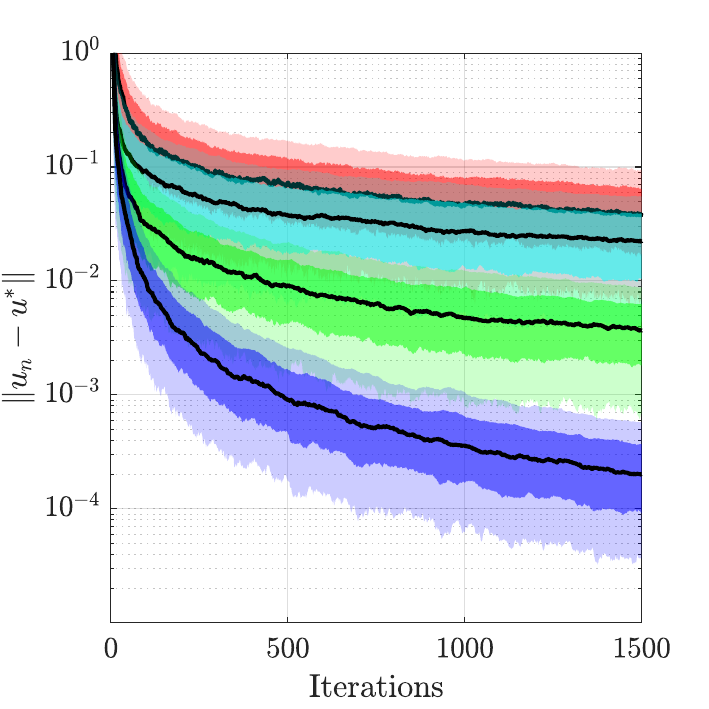}
    \end{minipage}\hfill
    \begin{minipage}[b][][c]{0.48\textwidth}
        \centering
        \includegraphics[width = \textwidth]{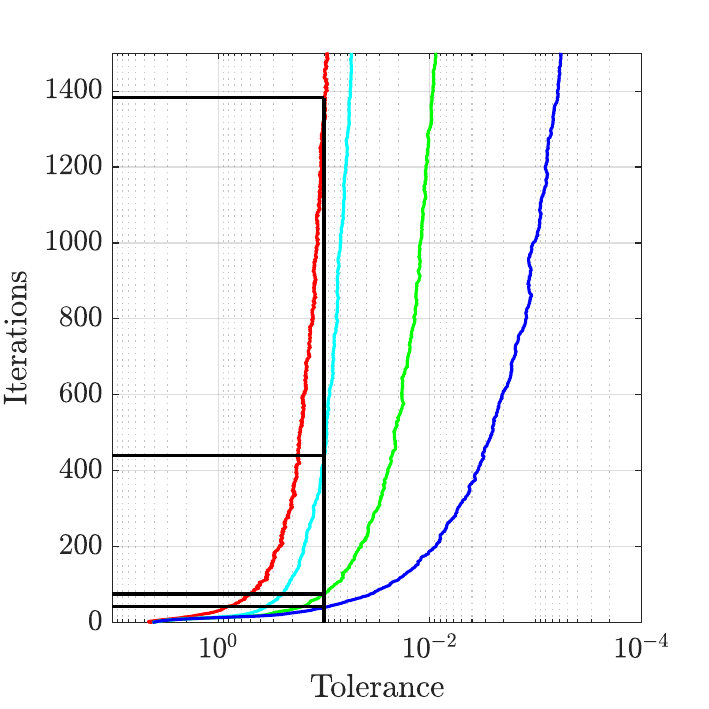}%
    \end{minipage}
    \par
    \begin{minipage}[t][][c]{0.48\textwidth}
        \caption{Absolute error $\Vert u_n-u^\ast\Vert$ during the optimization process. From top to bottom: aSCGD (red), CSG with empirical weights (cyan), CSG with inexact hybrid weights with $f(n) = \lfloor n^{1.5}\rfloor$ (green) and CSG with exact hybrid weights (blue). The shaded areas indicate the quantiles $P_{0.1,0.9}$ (light) and $P_{0.25,0.75}$ (dark).}
        \label{fig:SCGDCSG}
    \end{minipage}\hfill
    \begin{minipage}[t][][c]{0.48\textwidth}
        \caption{Minimum number of steps needed for the different algorithms such that at least 90\% of the runs achieve an absolute error smaller than the given tolerance. the colors are chosen in the same fashion as in \Cref{fig:SCGDCSG}. The exact numbers for a tolerance of $10^{-1}$ are 42 (exact hybrid), 76 (inexact hybrid), 440 (empirical) and 1382 (aSCGD).}
        \label{fig:SCGDCSG2}
    \end{minipage}
\end{figure}

\section{Convergence Results For Constant Step Size}\label{sec:ConvergenceConstant}
Our first result considers the special case in which the objective function $J$ appearing in \eqref{eq:Problemstellung} has only finitely many stationary points on $\U$. The proof of this result serves as a prototype for the later convergence results, as they share a common idea.
\begin{theorem}[Convergence for Constant Steps]\label{theo:ConstantStepConvergence}
Assume that $J$ has only finitely many stationary points on $\U$. 

Then CSG with a positive constant step size $\tau_n=\tau<\tfrac{2}{L}$ converges to a stationary point of $J$.
\end{theorem}
We want to sketch the proof of \Cref{theo:ConstantStepConvergence}, before going into details. In the deterministic case, \Cref{lem:descent lemma} and \Cref{lem: char property} are used to show that ${J(u_{n+1})\le J(u_n)}$ for all $n\in\N$. It then follows from a telescopic sum argument that $\Vert u_{n+1}-u_n\Vert\to0$, i.e., every accumulation point of $(u_n)_{n\in\N}$ is stationary (compare \cite[Theorem 5.1]{ProjGradientRates} or \cite[Theorem 10.15]{DescentUndProjection}).

In the case of CSG, we can not guarantee monotonicity of the objective function values $(J(u_n))_{n\in\N}$. Instead, we split the sequence into two subsequences. On one of these subsequences, we can guarantee a decrease in function values, while for the other sequence we can not. However, we prove that the latter sequence can accumulate only at stationary points of $J$. 
The main idea is then that $(u_n)_{n\in\N}$ can have only one accumulation point, because ``switching" between several points conflicts with the decrease in function values that must happen for steps in between.
\begin{proof}[Proof of \Cref{theo:ConstantStepConvergence}]
By \Cref{lem:descent lemma} we have
\begin{align*}
    &J(u_{n+1})-J(u_n) \le \nabla J(u_n)^\top \left( u_{n+1}-u_n\right) +\tfrac{L}{2}\Vert u_{n+1}-u_n\Vert_\sU^2 \\
    &\qquad= \hat{G}_n^\top\left(u_{n+1}-u_n\right)+\tfrac{L}{2}\Vert u_{n+1}-u_n\Vert_\sU^2 +\left(\nabla J(u_n)-\hat{G}_n\right)^\top\left(u_{n+1}-u_n\right).
\end{align*}
Utilizing \Cref{lem: char property} and the Cauchy-Schwarz inequality, we now obtain
\begin{align}
    J(u_{n+1}&)-J(u_n) \nonumber\\
    &\le \left(\tfrac{L}{2}-\tfrac{1}{\tau}\right)\Vert u_{n+1}-u_n\Vert_\sU^2+\left\Vert\nabla J(u_n)-\hat{G}_n\right\Vert\cdot\Vert u_{n+1}-u_n\Vert_\sU \nonumber\\
    &= \left( \left(\tfrac{L}{2}-\tfrac{1}{\tau}\right)\Vert u_{n+1}-u_n\Vert_\sU+\left\Vert\nabla J(u_n)-\hat{G}_n\right\Vert\right)\Vert u_{n+1}-u_n\Vert_\sU. \label{eq:AbschFunktionswerte}
\end{align}
Since $\frac{L}{2}-\frac{1}{\tau}<0$, our idea is the following:

Steps that satisfy 
\begin{equation}
    \left\Vert\nabla J(u_n)-\hat{G}_n\right\Vert\le\tfrac{1}{2}\left(\tfrac{1}{\tau}-\tfrac{L}{2}\right)\Vert u_{n+1}-u_n\Vert_\sU, \label{eq:ConditionSteps}
\end{equation}
i.e., steps with small errors in the gradient approximation, will yield decreasing function values. 

On the other hand, the remaining steps will satisfy $\Vert u_{n+1}-u_n\Vert_\sU\to0$, due to $\Vert\nabla J(u_n)-\hat{G}_n\Vert\to0$ (see \Cref{lem:GradientError}). With this in mind, we distinguish three cases: In Case 1, \eqref{eq:ConditionSteps} is satisfied for almost all steps, while in Case 2 it is satisfied for only finitely many steps. In the last case, there are infinitely many steps satisfying and infinitely many steps violating \eqref{eq:ConditionSteps}.
\begin{description}
\item[\textbf{Case 1:}]There exists $N\in\N$ such that 
\begin{equation*}
    \left\Vert\nabla J(u_n)-\hat{G}_n\right\Vert\le\tfrac{1}{2}\left(\tfrac{1}{\tau}-\tfrac{L}{2}\right)\Vert u_{n+1}-u_n\Vert_\sU \quad\text{for all }n\ge N.
\end{equation*}
In this case, it follows from \eqref{eq:AbschFunktionswerte} that $J(u_{n+1})\le J(u_n)$ for all $n\ge N$. Therefore, the sequence $(J(u_n))_{n\in\N}$ is monotonically decreasing for almost every $n\in\N$. Since $J$ is continuous and $\U$ is compact, $J$ is bounded and we therefore have ${J(u_n)\to\bar{J}\in\R}$. Thus, it holds
\begin{equation*}
    -\infty < \bar{J}-J(u_N) = \sum_{n=N}^\infty\big( J(u_{n+1})-J(u_n)\big) \le \tfrac{1}{2}\left(\tfrac{L}{2}-\tfrac{1}{\tau}\right)\sum_{n=N}^\infty\Vert u_{n+1}-u_n\Vert_\sU^2.
\end{equation*}
Since $\frac{1}{2}\left(\frac{L}{2}-\frac{1}{\tau}\right)<0$, we must have $\Vert u_{n+1}-u_n\Vert_\sU\to0$. Let $\left(u_{n_k}\right)_{k\in\N}$ be a convergent subsequence with $u_{n_k}\to\overline{u}\in\U$. 

By \Cref{lem:GradientError}, we have $\hat{G}_{n_k}\to\nabla J(\overline{u})$ and thus
\begin{align*}
    0 &= \lim_{k\to\infty}\Vert u_{n_k +1}-u_{n_k}\Vert_\sU \\
    &= \lim_{k\to\infty}\left\Vert \P_\U\big(u_{n_k}-\tau\nabla \hat{G}_{n_k}\big)-u_{n_k}\right\Vert_\sU \\
    &= \left\Vert \P_\U\big(\overline{u}-\tau\nabla J(\overline{u})\big)-\overline{u}\right\Vert_\sU,
\end{align*}
i.e., every accumulation point of $(u_n)_{n\in\N}$ is stationary. Since $J$ has only finitely many stationary points, \Cref{lem:Convergence} yields the convergence of $(u_n)_{n\in\N}$ to a stationary point of $J$.
\item[\textbf{Case 2:}] There exists $N\in\N$ such that 
\begin{equation*}
    \left\Vert\nabla J(u_n)-\hat{G}_n\right\Vert>\tfrac{1}{2}\left(\tfrac{1}{\tau}-\tfrac{L}{2}\right)\Vert u_{n+1}-u_n\Vert_\sU \quad\text{for all }n\ge N.
\end{equation*}
By \Cref{lem:GradientError}, we have $\Vert\nabla J(u_n)-\hat{G}_n\Vert\to0$. Since $\frac{1}{2}\left(\frac{1}{\tau}-\frac{L}{2}\right)>0$, the above inequality directly implies $\Vert u_{n+1}-u_n\Vert_\sU\to0$. Analogously to Case 1, we conclude that $(u_n)_{n\in\N}$ converges to a stationary point of $J$. 
\item[\textbf{Case 3:}] There are infinitely many $n\in\N$ with
\begin{equation*}
    \left\Vert\nabla J(u_n)-\hat{G}_n\right\Vert\le\tfrac{1}{2}\left(\tfrac{1}{\tau}-\tfrac{L}{2}\right)\Vert u_{n+1}-u_n\Vert_\sU
\end{equation*}
and infinitely many $n\in\N$ with
\begin{equation*}
    \left\Vert\nabla J(u_n)-\hat{G}_n\right\Vert>\tfrac{1}{2}\left(\tfrac{1}{\tau}-\tfrac{L}{2}\right)\Vert u_{n+1}-u_n\Vert_\sU.
\end{equation*}
In this case, we split $(u_n)_{n\in\N}$ disjointly in the two sequences $(u_{a(n)})_{n\in\N}$ and $(u_{b(n)})_{n\in\N}$, such that we have
\begin{equation*}
    \left\Vert\nabla J(u_{a(n)})-\hat{G}_{a(n)}\right\Vert\le\tfrac{1}{2}\left(\tfrac{1}{\tau}-\tfrac{L}{2}\right)\Vert u_{a(n)+1}-u_{a(n)}\Vert_\sU\quad \text{for all }n\in\N
\end{equation*}
and
\begin{equation*}
    \left\Vert\nabla J(u_{b(n)})-\hat{G}_{b(n)}\right\Vert>\tfrac{1}{2}\left(\tfrac{1}{\tau}-\tfrac{L}{2}\right)\Vert u_{b(n)+1}-u_{b(n)}\Vert_\sU\quad \text{for all }n\in\N.
\end{equation*}
We call $(u_{a(n)})_{n\in\N}$ the sequence of descent steps. For $(u_{b(n)})_{n\in\N}$, observe that, as in Case 2, we directly obtain 
\begin{equation}
    \Vert u_{b(n)+1}-u_{b(n)}\Vert_\sU\to0 \label{eq:StepsDecrease}
\end{equation} and every accumulation point of $(u_{b(n)})_{n\in\N}$ is stationary. Therefore, as in the proof of \Cref{lem:Convergence}, for all $\varepsilon>0$ there exists $N\in\N$ such that 
\begin{equation}
    \min_{i\in\{1,\ldots,K\}}\Vert u_{b(n)}-\overline{u}_i\Vert_\sU < \varepsilon\quad \text{for all }n\ge N, \label{eq:Close to acc point}
\end{equation}
where $\overline{u}_1,\ldots,\overline{u}_K$ denote the $K\in\N$ accumulation points of $(u_{b(n)})_{n\in\N}$.\\
\ \\
Now, we prove by contradiction that $J(\overline{u}_1)=J(\overline{u}_2)=\ldots=J(\overline{u}_K)$.

Suppose that this is not the case. Then we have at least $M\ge2$ function values of accumulation points and
\begin{equation*}
    F:=\left\{ J(u)\,:\, u = \overline{u}_1,\ldots,\overline{u}_K\right\} = \{f_1,f_2,\ldots,f_M\}
\end{equation*}
for some $f_1>f_2>\ldots>f_M\in\R$. Now, choose $\varepsilon>0$ small enough, such that 
\begin{equation*}
    2\varepsilon < \min_{\substack{i,j\in\{1,\ldots,K\} \\ i\neq j}} \Vert \bar{u}_i-\bar{u}_j\Vert\quad\text{and}\quad c_{_L}\varepsilon<f_1-f_2,
\end{equation*}
where $c_{_L}$ denotes the Lipschitz constant of $J$. By \eqref{eq:StepsDecrease} and \eqref{eq:Close to acc point}, there exists $N\in\N$ such that for all $n\ge N$ we have
\begin{equation}
    \Vert u_{b(n)+1}-u_{b(n)}\Vert_\sU<\tfrac{\varepsilon}{4} \quad\text{and}\quad\min_{i\in\{1,\ldots,K\}}\Vert u_{b(n)}-\overline{u}_i\Vert_\sU < \tfrac{\varepsilon}{4}.\label{eq:BnBeiHP}
\end{equation}
Therefore, for $n\ge N$ and $i\in\{1,\ldots,K\}$, we have
\begin{align}
    u_{b(n)}\in\mathcal{B}_{\frac{\varepsilon}{4}}(\overline{u}_i) \; &\Longrightarrow\; u_{b(n)+1}\in\mathcal{B}_{\frac{\varepsilon}{2}}(\overline{u}_i) \label{eq:Eps4ZuEps2}\\
    &\Longrightarrow\; u_{b(n)+1}\not\in\mathcal{B}_{\frac{\varepsilon}{4}}(\overline{u}_j) \text{ for all $j\neq i$.} \label{eq:BnRemainsClose}
\end{align}
Especially, for all $n\ge N$ and all $i=1,\ldots, K$ it holds
\begin{equation}
    u_{b(n)}\in\mathcal{B}_{\frac{\varepsilon}{4}}(\overline{u}_i)\;\Longrightarrow\; J(u_{b(n)+1}) \le J(\overline{u}_i)+\tfrac{c_{_L}\varepsilon}{2}. \label{eq:FunctionValueDifference}
\end{equation}
It follows from \eqref{eq:BnBeiHP} and \eqref{eq:BnRemainsClose} that for $n\ge N+1$:
\begin{description}
    \item[(A)] If $u_{b(n)}\in\mathcal{B}_{\frac{\varepsilon}{4}}(\overline{u}_i)$ and $u_{b(n+1)}\in\mathcal{B}_{\frac{\varepsilon}{4}}(\overline{u}_j)$ for some $j\neq i$, then there must be at least one descent step between $u_{b(n)}$ and $u_{b(n+1)}$.
    \item[(B)] If $u_{b(n)}\in\mathcal{B}_{\frac{\varepsilon}{4}}(\overline{u}_i)$ and $u_{b(n)-1}\not\in\mathcal{B}_{\frac{\varepsilon}{4}}(\overline{u}_i)$, then $u_{b(n)-1}$ must be a descent step. 
\end{description}
Observe that (A) follows directly from \eqref{eq:BnBeiHP} and \eqref{eq:BnRemainsClose}, as moving from the vicinity of $\overline{u}_i$ to a neighborhood of $\overline{u}_j$ requires that there is an intermediate step $u_n$ with $\min_{i\in\{1,\ldots,K\}}\Vert u_{b(n)}-\overline{u}_i\Vert_\sU \ge \tfrac{\varepsilon}{4}$. Similarly, (B) is just the second condition in \eqref{eq:BnRemainsClose} reformulated.

Now, let $i\in\{1,\ldots,K\}$ be chosen such that $J(\overline{u}_i)\le f_2$ and let $n_0\ge N$ be chosen such that $u_{b(n_0)}\in\mathcal{B}_{\frac{\varepsilon}{4}}(\overline{u}_i)$ and $u_{b(n_0)+1}\not\in\mathcal{B}_{\frac{\varepsilon}{4}}(\overline{u}_i)$. Using \eqref{eq:Eps4ZuEps2} and \eqref{eq:FunctionValueDifference}, we obtain
\begin{align*}
    J(u_{b(n_0)+1}) &\le J(\overline{u}_i)+\frac{c_{_L}\varepsilon}{2} \\
        &\le f_2+\frac{c_{_L}\varepsilon}{2} \\
        &< f_1-\frac{c_{_L}\varepsilon}{2}\\
        &< J(u)\quad\text{for all }u\in\mathcal{B}_{\frac{\varepsilon}{4}}\left( J^{-1}(\{f_1\})\cap\{\overline{u}_1,\ldots,\overline{u}_K\}\right).
\end{align*}
Therefore, descent steps can never reach $\mathcal{B}_{\frac{\varepsilon}{4}}\left( J^{-1}(\{f_1\}\cap\{\overline{u}_1,\ldots,\overline{u}_K\})\right)$ again! It follows from item (B), that $u_n\not\in\mathcal{B}_{\frac{\varepsilon}{4}}\left( J^{-1}(\{f_1\})\cap\{\overline{u}_1,\ldots,\overline{u}_K\}\right)$ for all ${n\ge b(n_0)}+1$, in contradiction to $J^{-1}(\{f_1\})\cap\{\overline{u}_1,\ldots,\overline{u}_K\}$ consisting of at least one accumulation point of $(u_n)_{n\in\N}$. Hence, we have
\begin{equation}
    J(\overline{u}_1) = \ldots = J(\overline{u}_K) =: \bar{J}.\label{eq:FunctionValuesIdentical}
\end{equation}
Next, we show that every accumulation point of $(u_{a(n)})_{n\in\N}$ is stationary. We prove this by contradiction.\\
Assume there exists a non-stationary accumulation point $\overline{u}$ of $(u_{a(n)})_{n\in\N}$. Observe that
\begin{equation*}
    \min_{i\in\{1,\ldots,K\}}\Vert\overline{u}-\overline{u}_i\Vert_\sU > 0.
\end{equation*}
\begin{description}
    \item[\textbf{Case 3.1:}]$J(\overline{u})<\bar{J}$.
    
    Then, by the same arguments as above, there exists $N\in\N$ and $\e>0$ s.t. 
    \begin{equation*}
    u_n\not\in\bigcup_{i=1}^K\mathcal{B}_{\frac{\varepsilon}{4}}\left(\overline{u}_i\right)\quad\text{for all } n\ge\N.
    \end{equation*}
    This is a contradiction to $\overline{u}_1,\ldots,\overline{u}_K$ being accumulation points of $(u_n)_{n\in\N}$.
    \item[\textbf{Case 3.2:}] $J(\overline{u})>\bar{J}$.
    
    In this case, there exists $N\in\N$ and $\e>0$ such that $u_n\not\in\mathcal{B}_{\frac{\varepsilon}{4}}\left(\overline{u}\right)$ for all $n\ge N$. This is a contradiction to $\overline{u}$ being an accumulation point of $(u_n)_{n\in\N}$.
    \item[\textbf{Case 3.3:}] $J(\overline{u})=\bar{J}$.
    
    Since $\overline{u}$ is an accumulation point of $(u_{a(n)})_{n\in\N}$, there exists a subsequence $(u_{a(n_k)})_{k\in\N}$ with $u_{a(n_k)}\to\overline{u}$. The sequence $(u_{a(n_k)-1})_{k\in\N}$ is bounded and therefore has at least one accumulation point $\overline{u}_{-1}$ and a subsequence $(u_{a(n_{k_t})-1})_{t\in\N}$ with $u_{a(n_{k_t})-1}\to\overline{u}_{-1}$. It follows that
    \begin{align*}
       \P_\U\left(\overline{u}_{-1}-\tau\nabla J(\overline{u}_{-1})\right) &= \lim_{t\to\infty} \P_\U\left(u_{a(n_{k_t})-1}-\tau \hat{G}_{a(n_{k_t})-1}\right) \\
       &= \lim_{t\to\infty} u_{a(n_{k_t})} \\
       &= \overline{u}.
    \end{align*}
    As $\overline{u}$ is not stationary by our assumption, $\overline{u}_{-1}\neq\overline{u}$ and $\overline{u}_{-1}$ is no stationary point of $J$. Thus, \Cref{lem:descent lemma} combined with \Cref{lem: char property} yields
    \begin{equation*}
        J(\overline{u})-J(\overline{u}_{-1})\le \left(\tfrac{L}{2}-\tfrac{1}{\tau}\right)\Vert \overline{u}_{-1}-\overline{u}\Vert_\sU^2 < 0.
    \end{equation*}
    Therefore, $\overline{u}_{-1}$ is an accumulation point of $(u_{a(n)})_{n\in\N}$, which satisfies  ${J(\overline{u}_{-1})>J(\overline{u})=\bar{J}}$. This, however, is impossible, as seen in Case 3.2.
\end{description}
In conclusion, in Case 3, all accumulation points of $(u_n)_{n\in\N}$ are stationary. Thus, on every convergent subsequence we have $\Vert u_{n_k+1}-u_{n_k}\Vert_\sU\to 0$. Since $(u_n)_{n\in\N}$ is bounded, this already implies $\Vert u_{n+1}-u_n\Vert_\sU\to 0$. Now, \Cref{lem:Convergence} yields the claimed convergence of $(u_n)_{n\in\N}$ to a stationary point of $J$.
\end{description}
\end{proof}
The idea of the proof above still applies in the case that $J$ is constant on some parts of $\U$, i.e., $J$ can have infinitely many stationary points. We obtain the following convergence result:
\begin{theorem}\label{cor:ConstantStepConvergence}
Let $\mathcal{S}(J)$ be the set of stationary points of $J$ on $U$ as defined in \Cref{def:StationaryPoints}. Assume that the set
\begin{equation*}
    \mathcal{N}:= \big\{ J(u)\,:\, u\in\mathcal{S}(J)\big\}\subset\R
\end{equation*}
is of Lebesgue-measure zero. Then every accumulation point of the sequence generated by CSG with constant step size $\tau<\tfrac{2}{L}$ is stationary and we have convergence in function values. 
\end{theorem}
\begin{remark} Comparing \Cref{theo:ConstantStepConvergence} and \Cref{cor:ConstantStepConvergence}, observe that under the weaker assumptions on the set of stationary points of $J$, we no longer obtain convergence for the whole sequence of iterates. To illustrate why that is the case, consider the function $J:\R^\ddes\to\R$ given by $J(u)=\cos(\pi\Vert u\Vert_2^2)$ and $\U=\{ u\in\R^\ddes\,:\,\Vert u\Vert_2^2\le \frac{3}{2}\}$. Then, $\mathcal{S}(J)=\{0\}\cup\{u\in\U\,:\,\Vert u\Vert_2 =1\}$, i.e., every point on the unit sphere is stationary. Thus, we can not use \Cref{lem:Convergence} at the end of the proof to obtain convergence of $(u_n)_{n\in\N}$. Theoretically, it might happen that the iterates $(u_n)_{n\in\N}$ cycle around the unit sphere, producing infinitely many accumulation points, all of which have the same objective function value. This, however, did not occur when testing this example numerically.
\end{remark}
\begin{remark}\label{rem:Sard}
While the assumption in \Cref{cor:ConstantStepConvergence} seems unhandy at first, there is actually a rich theory concerning such properties. For example, Sard's Theorem \cite{SardTheo} and generalizations \cite{SardAllgemein} give that the assumption holds if $J\in C^\ddes$ and $\U$ has smooth boundary. Even though it can be shown that there exist functions, which do not satisfy the assumption (e.g. \cite{AntiSard1, AntiSard2}), such counter-examples need to be precisely constructed and will most likely not appear in any application.
\end{remark}

\begin{proof}[Proof of \Cref{cor:ConstantStepConvergence}]
Proceeding analogously as in the proof of \Cref{theo:ConstantStepConvergence}, we only have to adapt two intermediate results in Case 3: 
\begin{description}
    \item[(R1)] The objective function values of all accumulation points of $(u_{b(n)})_{n\in\N}$ are equal.
    \item[(R2)] Every accumulation point of $(u_{a(n)})_{n\in\N}$ is stationary.
\end{description}
Assume first, that (R1) does not hold. Then there exist two stationary points ${\overline{u}_1\neq\overline{u}_2}$ with $J(\overline{u}_1)< J(\overline{u}_2)$. Now, (A) and (B) shown in the proof of \Cref{theo:ConstantStepConvergence} yield that there must exist an accumulation point $\overline{u}_3$ of $(u_{b(n)})_{n\in\N}$, i.e., a stationary point, with $J(\overline{u}_1)<J(\overline{u}_3)<J(\overline{u}_2)$. Iterating this procedure, we conclude that the set $\mathcal{N}\cap\big[ J(\overline{u}_1),J(\overline{u}_2)\big]$ is dense in $\big[ J(\overline{u}_1),J(\overline{u}_2)\big]$. 

By continuity of $u\mapsto\P_\U\big(u-\tau\nabla J(u)\big)-u$ and compactness of $\U$, we see that 
\begin{equation*}
    \mathcal{N}\cap\big[ J(\overline{u}_1),J(\overline{u}_2)\big] = \big[ J(\overline{u}_1),J(\overline{u}_2)\big],
\end{equation*}
contradicting our assumption that $\lambda(\mathcal{N})=0$.

For (R2), assume that $(u_{a(n)})_{n\in\N}$ has a non-stationary accumulation point $\overline{u}$. Since $\mathcal{S}(J)$ is compact, it holds 
\begin{equation*}
    \dist(\{\overline{u}\},\mathcal{S}(J)) > 0.
\end{equation*}
Thus, by the same arguments as in Case 3.1, 3.2 and 3.3 within the proof of \Cref{theo:ConstantStepConvergence}, we observe that such a point $\overline{u}$ can not exist. 
\end{proof}

\subsection{Academic Example for Constant Step Size}\label{subsec:ExampleConstSteps}
Define $\U=\left[-\frac{1}{2},\frac{1}{2}\right]$, $\X=\left(-\frac{1}{2},\frac{1}{2}\right)$ and consider the problem
\begin{equation}
    \min_{u\in\U}\; \frac{1}{2}\int_\X (u-x)^2\dd x. \label{eq:ProblemConstSteps}
\end{equation}
It is easy to see that \eqref{eq:ProblemConstSteps} has a unique solution $u^\ast = 0$. Furthermore, the objective function is $L$-smooth (with Lipschitz constant 1) and even strictly convex. Thus, by \Cref{theo:ConstantStepConvergence}, the CSG method with a constant positive step size $\tau < 2$ produces a sequence $(u_n)_{n\in\N}$ that satisfies $u_n\to0$. 

However, even in this highly regular setting, the commonly used basic SG method does not guarantee convergence of the iterates for a constant step size. 

To demonstrate this behavior of both CSG and SG, we draw 2000 random starting points $u_0\in\U$ and compare the iterates produced by CSG and SG with five different constant step sizes ($\tau\in\{0.01,0.1,1,1.9,1.99\}$). The CSG integration weights were calculated using the empirical method, i.e., the computationally cheapest choice. The results are shown in \Cref{fig:ConstSteps}.

As expected, the iterates produced by the SG method do not converge to the optimal solution, but instead remain in a neighborhood of $u^\ast$. The radius of said neighborhood depends on the choice of $\tau$ and decreases for smaller $\tau$, see \cite[Theorem 4.6]{Steps01}.
\begin{figure}
    \centering
    
        {\includegraphics[scale=.43,clip,trim=0 38 5 0]{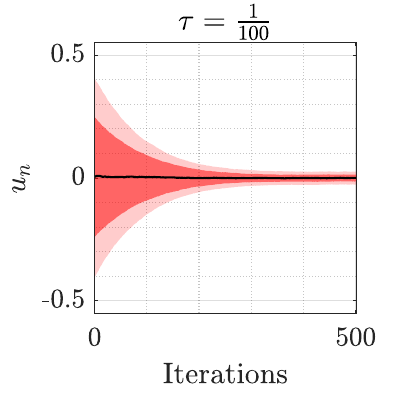}}
    {\includegraphics[scale=.43,clip,trim=41 38 5 0]{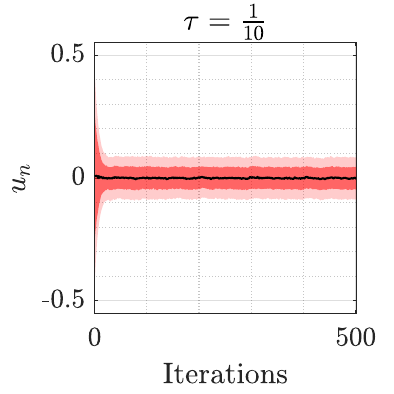}}
    {\includegraphics[scale=.43,clip,trim=41 38 5 0]{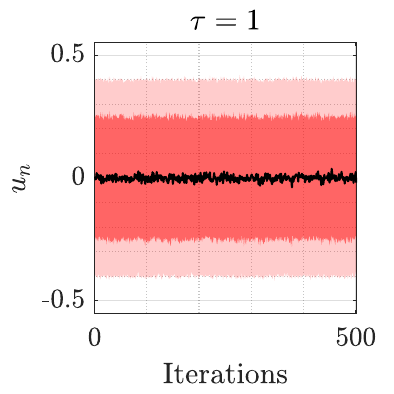}}
    {\includegraphics[scale=.43,clip,trim=41 38 5 0]{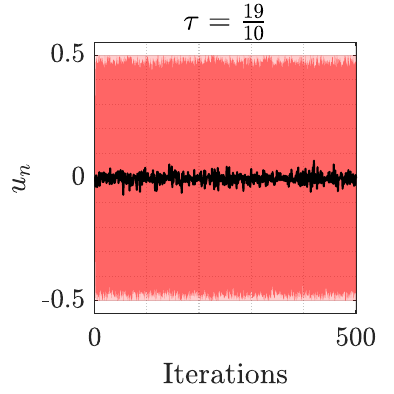}}
    {\includegraphics[scale=.43,clip,trim=41 38 5 0]{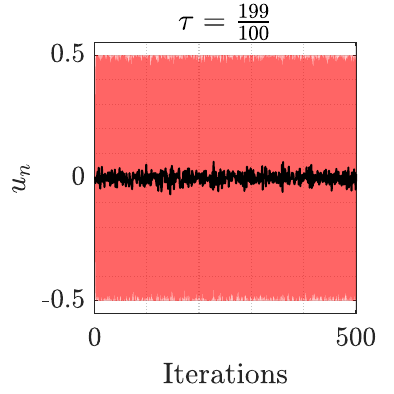}}

    {\includegraphics[scale=.43,clip,trim=0 0 5 20]{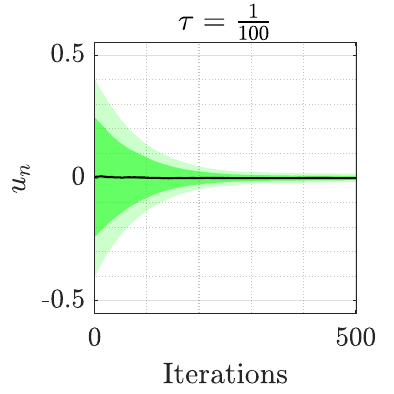}}
    {\includegraphics[scale=.43,clip,trim=41 0 5 20]{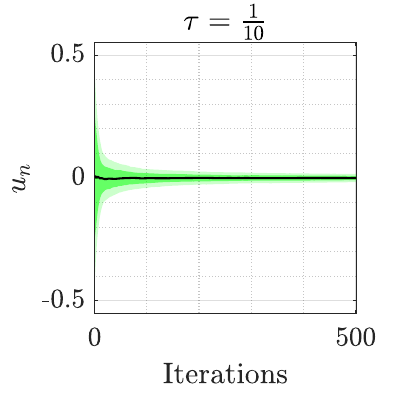}}
    {\includegraphics[scale=.43,clip,trim=41 0 5 20]{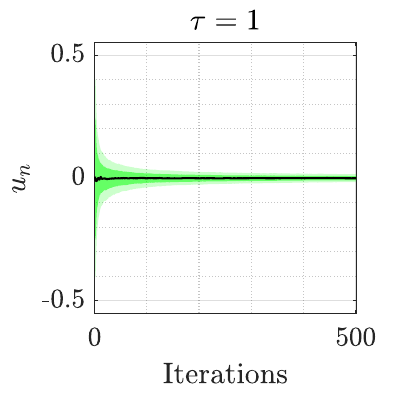}}
    {\includegraphics[scale=.43,clip,trim=41 0 5 20]{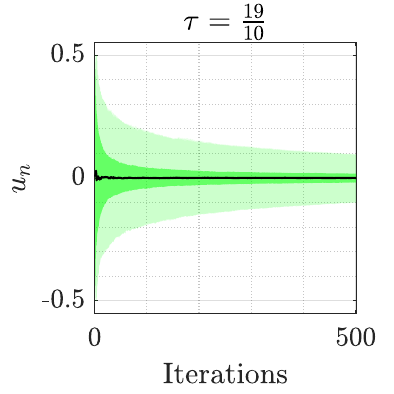}}
    {\includegraphics[scale=.43,clip,trim=41 0 5 20]{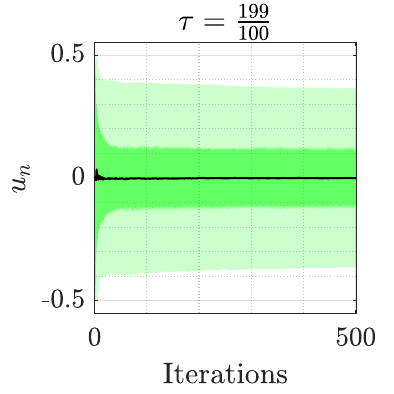}}
    
    \caption{Comparison of the iterates produced by 500 steps of SG (red / first row) and CSG (green / second row) with 2000 random starting points $u_0\in\U$. Both methods have been tested for the five different constant step sizes $\tau\in\{0.01,0.1,1,1.9,1.99\}$ (first to fifth column). The shaded areas indicate the quantiles $P_{0.1,0.9}$ (light) and $P_{0.25,0.75}$ (dark), while the solid line represents the median of the 2000 runs.}
    \label{fig:ConstSteps}
\end{figure}
\section{Backtracking}\label{sec:ConvergenceArmijo}
As $\Vert\hat{G}_n-\nabla J(u_n)\Vert \to 0$ and $\Vert\hat{J}_n- J(u_n)\Vert \to 0$ for $n\to\infty$, we can use these approximations to refine the steplength by a backtracking line search method.

\begin{definition}
For simplicity, we define
\begin{equation*}
    s_n(t):= \P_\U\left(u_n-t\hat{G}_n\right).
\end{equation*}
Furthermore, given $n$ gradient samples $\nabla_1 j(u_i,x_i)$ and $n$ cost function samples $j(u_i,x_i)$, by calculating the weights $\alpha_i^{(n)}(u)$ w.r.t. a given point $u\in\U$, we define
\begin{equation*}
    \tilde{J}_n(u) = \sum_{i=1}^n \alpha_i^{(n)}(u)j(u_i,x_i)\quad\text{ and }\quad\tilde{G}_n(u) = \sum_{i=1}^n \alpha_i^{(n)}(u)\nabla_1j(u_i,x_i),
\end{equation*}
which are approximations to $J(u)$ and $\nabla J(u)$ respectively.
\end{definition}
Based on the well known Armijo-Wolfe conditions from continuous optimization \cite{ArmijoCondition,WolfeConditions1,WolfeConditions2}, we introduce the following step size conditions:
\begin{definition} For $0<c_1<c_2<1$, we call $s_n(\tau_n)$ an Armijo step, if 
\begin{align}
    &\tilde{J}_n(s_n(\tau_n)) \le \hat{J}_n - c_1\hat{G}_n^\top\left(u_n - s_n(\tau_n)\right). \tag{SW1} \label{eq:SW1} \\
    \intertext{Additionally, we define the following Wolfe-type condition:}
    &\tilde{G}_n(s_n(\tau_n))^\top\left(s_n(\tau_n)-u_n \right) \ge c_2\hat{G}_n^\top\left(s_n(\tau_n)-u_n \right) . \tag{SW2} \label{eq:SW2}
\end{align}
\end{definition}

We try to obtain a step size that satisfies \eqref{eq:SW1} and \eqref{eq:SW2} by a bisection approach, as formulated in \Cref{alg:Backtracking}. Since we can not guarantee to find a suitable step size, we perform only a fixed number $T\in\N$ of backtracking steps. Notice that the curvature condition \eqref{eq:SW2} only has an influence, if $u_n-\tau_n\hat{G}_n\in\U$ (see line 6 in \Cref{alg:Backtracking}). This way, we gain the advantages of a Wolfe line search while inside $\U$, without ruling out stationary points at the boundary of $\U$. 

For our convergence analysis, we assume that in each iteration of CSG with line search, \Cref{alg:Backtracking} is initiated with the same $\eta_n=\eta>0$. From a practical point of view, we might also consider a diminishing sequence $(\eta_n)_{n\in\N}$ of backtracking initializations (see \Cref{sec:StepsizeStabilityEx}). The CSG method with backtracking line search (bCSG) is given in \Cref{alg:CSGBacktracking}.
\begin{algorithm} 
    \caption{Backtracking Refinement}
    \label{alg:Backtracking}
    \begin{algorithmic}[0]
    \State{Given $T\in\N$, $0<c_1<c_2<1$ appearing in \eqref{eq:SW1Star} and \eqref{eq:SW2}, $u_n\in\U$, and $\eta > 0$, \\ set $t=1$, $a=0$, $b=\infty$, $\eta_A = \infty$.}
    \While{$t \le T$}
    \State{Calculate step $s=\mathcal{P}_{\U}(u_n-\eta\hat{G}_n)$, weights $\alpha_k^{(n)}(s)$ and $\tilde{J}_n(s)$, $\tilde{G}_n(s)$;}
    \If{\eqref{eq:SW1Star} is not satisfied}
    \State{$b=\eta$;}
    \ElsIf{$s = u_n-\eta\hat{G}_n$ \textbf{and} \eqref{eq:SW2} is not satisfied }
    \State{$a=\eta$;}
    \State{$\eta_A = \eta$;}
    \Else
    \State{\textbf{break}}
    \EndIf
    \If{$b<\infty$}
    \State{$\eta = \frac{a+b}{2}$;}
    \Else
    \State{$\eta = 2a$;}
    \EndIf
    \State{$t = t+1$;}
    \EndWhile
    \If{$t=T+1$ and $\eta_A <\infty$}
    \State{$\tau_n = \eta_A$;}
    \Else
    \State{$\tau_n = \eta$;}
    \EndIf
    \end{algorithmic}
\end{algorithm}
\begin{algorithm}
    \caption{Backtracking CSG (bCSG)}
    \label{alg:CSGBacktracking}
    \begin{algorithmic}[1]
    \State{Given $u_0\in\U$, and a positive sequence $(\eta_n)_{n\in\N}$,}
    \While{Termination condition not met}
    \State{Sample objective function:\\
            $\qquad j_n := j(u_n,x_n)$}
    \State{Sample gradient:\\
            $\qquad g_n := \nabla_1 j(u_n,x_n)$}
    \State{Calculate weights $\alpha_k$}
    \State{Calculate search direction:\\
            $\qquad \hat G_{n} := \sum_{k=1}^n  \alpha_k g_k \vphantom{\Big(}$}
    \State{Compute objective function value approximation:\\
            $\qquad \hat J_{n} := \sum_{k=1}^n  \alpha_k j_k \vphantom{\Big(}$} 
    \State{Calculate step size $\tau_n$ by \Cref{alg:Backtracking} with start at $\eta_n$.}
    \State{Gradient step:\\
            $\qquad u_{n+1} := \mathcal{P}_{\U}\big(u_n - \tau_n \hat G_{n}\big) \vphantom{\Big(}$}
    \State{Update index:\\
            $\qquad n = n+1 \vphantom{\Big(}$}
    \EndWhile
    \end{algorithmic}
\end{algorithm}
Since all of the terms $\tilde{J}_n(s_n(\tau_n))$, $\hat{J}_n$ and $\hat{G}_n$ appearing in \eqref{eq:SW1} contain some approximation error when compared to $J(s_n(\tau_n))$, $J(u_n)$ and $\nabla J(u_n)$ respectively, especially the first iterations of \Cref{alg:CSGBacktracking} might profit from a slightly weaker formulation of \eqref{eq:SW1}. Therefore, in practice, we will replace \eqref{eq:SW1} by the non-monotone version
\begin{equation}
    \tilde{J}_n(s_n(\tau_n)) \le \max_{k \in\{ 0,\ldots,K\}}\hat{J}_{n-k} - c_1\hat{G}_n^\top\left(u_n - s_n(\tau_n)\right), \label{eq:SW1Star} \tag{SW1$^\ast$}
\end{equation}
for some $K\in\{0,\ldots,n\}$. 

\subsection{Convergence Results}
For CSG with backtracking line search, we obtain the same convergence results as for constant step sizes:
\begin{theorem}[Convergence for Backtracking Line Search]\label{theo:ArmijoConvergence}
Let $\mathcal{S}(J)$ be the set of stationary points of $J$ on $U$ as defined in \Cref{def:StationaryPoints}. Assume that 
\begin{equation*}
    \mathcal{N}:= \big\{ J(u)\,:\, u\in\mathcal{S}(J)\big\}\subset\R
\end{equation*}
is of Lebesgue-measure zero and $T$ in \Cref{alg:Backtracking} is chosen large enough, such that $2^{-T}\eta<\frac{2}{L}$. Then every accumulation point of the sequence $(u_n)_{n\in\N}$ generated by \Cref{alg:CSGBacktracking} is stationary and we have convergence in function values.

If $J$ satisfies the stronger assumption of having only finitely many stationary points, $(u_n)_{n\in\N}$ converges to a stationary point of $J$.
\end{theorem}
\begin{proof}
Notice first, that there are only two possible outcomes of \Cref{alg:Backtracking}: Either $\tau_n$ satisfies \eqref{eq:SW1}, or $\tau_n=2^{-T}\eta<\frac{2}{L}$. Furthermore, as we have seen in the proof of \Cref{lem:Convergence}, for all $\varepsilon>0$ almost all $u_n$ lie in $\varepsilon$-Balls around the accumulation points of $(u_n)_{n\in\N}$, since $(u_n)_{n\in\N}$ is bounded. Therefore, $\tilde{J}_n(u_{n+1})- J(u_{n+1})\to0$ and $\tilde{G}_n(u_{n+1})-\nabla J(u_{n+1})\to0$ (compare \Cref{lem:GradientError}). Since we already know, that the steps with constant step size $\tau_n=2^{-T}\eta<\frac{2}{L}$ can be split in descent steps and steps which satisfy $\Vert u_{n+1}-u_n\Vert\to0$, we now take a closer look at the Armijo-steps, i.e., steps with $\tau_n\neq 2^{-T}\eta$.

If $\tau_n\neq 2^{-T}\eta$, by \eqref{eq:SW1} and \Cref{lem: char property}, it holds
\begin{align*}
    J(u_{n+1})-J(u_n) &\le -c_1\frac{\Vert u_{n+1}-u_n\Vert_\sU^2}{\tau_n^2}+\left\vert J(u_n)-\hat{J}_n\right\vert + \left\vert J(u_{n+1})-\tilde{J}_n(u_{n+1})\right\vert \\
    &\le -c_1\frac{\Vert u_{n+1}-u_n\Vert_\sU^2}{\tau_{\text{max}}^2}+\left\vert J(u_n)-\hat{J}_n\right\vert + \left\vert J(u_{n+1})-\tilde{J}_n(u_{n+1})\right\vert.
\end{align*}
Therefore, we either have
\begin{equation*}
    \left\vert J(u_n)-\hat{J}_n\right\vert + \left\vert J(u_{n+1})-\tilde{J}_n(u_{n+1})\right\vert \le c_1\frac{\Vert u_{n+1}-u_n\Vert_\sU^2}{\tau_{\text{max}}^2},
\end{equation*}
in which case it holds $J(u_{n+1})\le J(u_n)$, or 
\begin{equation*}
    \left\vert J(u_n)-\hat{J}_n\right\vert + \left\vert J(u_{n+1})-\tilde{J}_n(u_{n+1})\right\vert > c_1\frac{\Vert u_{n+1}-u_n\Vert_\sU^2}{\tau_{\text{max}}^2},
\end{equation*}
in which case $\tilde{J}_n(u_{n+1})-J(u_{n+1})\to0$ and $\hat{J}_n- J(u_{n})\to0$ yield $\Vert u_{n+1}-u_n\Vert_\sU\to0$.

Thus, regardless of whether or not $\tau_n=2^{-T}\eta$, we can split $(u_n)_{n\in\N}$ in a subsequence of descent steps and a subsequence of steps with $\Vert u_{n+1}-u_n\Vert_\sU\to0$. The rest of the proof is now identical to the proof of \Cref{theo:ConstantStepConvergence} and \Cref{cor:ConstantStepConvergence}.
\end{proof} 
\subsection{Step Size Stability for bCSG}\label{sec:StepsizeStabilityEx}

To analyze the proclaimed stability of bCSG with respect to the initially guessed step size $\eta_n$, we set $\U=[-10,10]^5$, $\X=(-1,1)^5$ and consider the Problem
\begin{equation}
    \min_{u\in\U}\; J(u), \label{eq:ProblemStability}
\end{equation}
where
\begin{equation*}
    J(u)=-\int_{\X}\frac{20}{1+\Vert u-x\Vert^2}\,\dd x.
\end{equation*}
Problem \eqref{eq:ProblemStability} has the unique solution $u^\ast=0\in\U$, which can be found analytically. 

As a comparison to our method, we choose the AdaGrad \cite{AdaGradPaper} algorithm, as it is widely used for problems of type \eqref{eq:Problemstellung}.
Both AdaGrad and bCSG start each iteration with a presdescribed step size $\eta_n>0$, based on which the calculation of the true step size $\tau_n$ is performed (see \Cref{alg:Backtracking}). We want to test the stability of both methods with respect to the initially chosen step length.
For this purpose, we set $\eta_n = \tfrac{\tau_0}{n^d}$, where $\tau_0>0$, $n$ is the iteration count and $d\in[0,1]$ is fixed. 

For each combination of $\tau_0$ and $d$, we choose 1200 random starting points in $\U$ and perform 500 optimization steps with both AdaGrad and backtracking CSG. Again, the integration weight calculation in bCSG was carried out using the empirical method, leading to a faster weight calculation while decreasing the overall progress per iteration performance. The median of the absolute error $\Vert u_{500}-u^\ast\Vert$ after the optimization, depending on $d$ and $\tau_0$, is presented in \Cref{fig:AdaStab}.

 While there are a few instances where AdaGrad yields a better result than backtracking CSG, we observe that the performance of AdaGrad changes rapidly, especially with respect to the parameter $d$. The backtracking CSG method on the other hand performs superior in most cases and is much less dependent on the choice of parameters.
\begin{figure}
    \centering
    \begin{minipage}[t][][c]{0.48\textwidth}
        \centering
        \includegraphics[width = \textwidth]{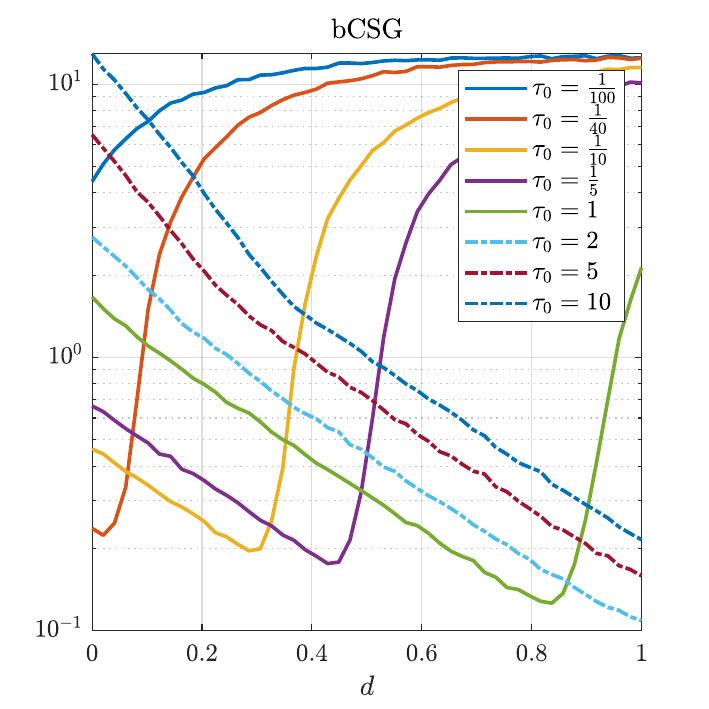}%
    \end{minipage}\hfill
    \begin{minipage}[t][][c]{0.48\textwidth}
        \centering
        \includegraphics[width = \textwidth]{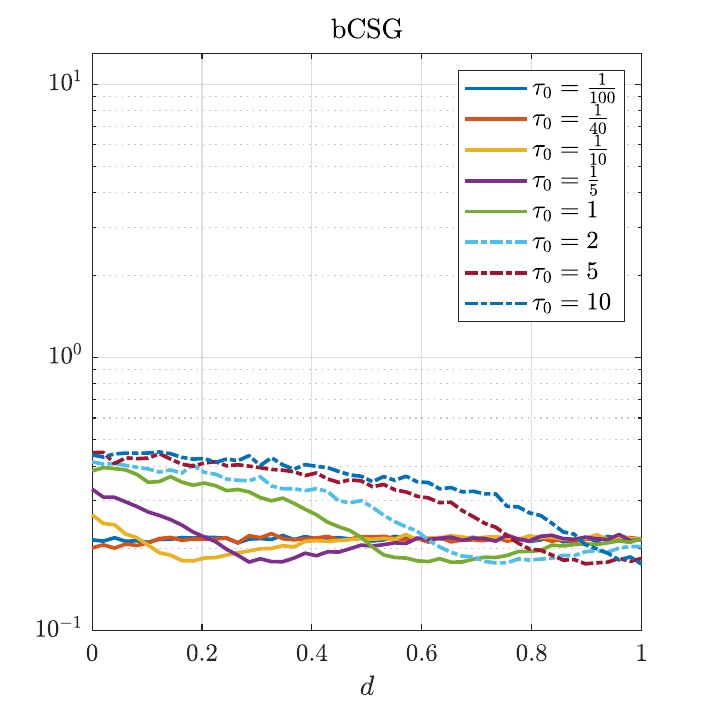}%
    \end{minipage}
    \caption{Median final error $\Vert u_{500} - u^\ast\Vert$ after 500 iterations of AdaGrad and bCSG. Depending on the constants $\tau_0>0$ and $d\in[0,1]$, the initially chosen step size in each iteration is given by $\eta_n =  \tau_0 n^{-d}$.   }
    \label{fig:AdaStab}
\end{figure}
\subsection{Estimations for the Lipschitz Constant of $\nabla J$}\label{sec:EstimateLip}
We have already seen, that the Lipschitz constant $L$ of $\nabla J$ is closely connected with efficient bounds on the step sizes. However, in general, we can not expect to have any knowledge of $L$ a priori. Thus, we are interested in an approximation of $L$, that can be calculated during the optimization process.

Investigating the proof of \Cref{lem:descent lemma} in \cite{DescentUndProjection}
\begin{align*}
    J(u_1) & = J(u_2) + \int_0^1\left\langle \nabla J(u_2)+t(u_1-u_2)),u_1-u_2\right\rangle \dd t \\
    &= J(u_2) + \left\langle\nabla J(u_2),u_1-u_2\right\rangle \\
        &\quad\quad+\int_0^1 \left\langle \nabla J\big(u_2+t(u_1-u_2)\big)-\nabla J(u_2),u_1-u_2\right\rangle \dd t\\
    &\le J(u_2) + \left\langle\nabla J(u_2),u_1-u_2\right\rangle \\
        &\quad\quad+\int_0^1 \left\Vert \nabla J\big(u_2+t(u_1-u_2)\big)-\nabla J(u_2)\right\Vert\cdot\left\Vert u_1-u_2\right\Vert_\sU \dd t \\
    &\le J(u_2) + \left\langle\nabla J(u_2),u_1-u_2\right\rangle + \int_0^1 Lt\Vert u_1-u_2\Vert_\sU^2\dd t \\
    &= J(u_2) + \left\langle\nabla J(u_2),u_1-u_2\right\rangle + \tfrac{L}{2}\Vert u_1-u_2\Vert_\sU^2,
\end{align*}
we observe that we do not need the true Lipschitz constant $L$ of $\nabla J$ for the second inequality. Instead, it is sufficient to choose any constant $C=C(u_1,u_2)$ that satisfies
\begin{equation*}
    \left\Vert\nabla J\big(u_2+t(u_1-u_2)\big)-\nabla J(u_2)\right\Vert \le C\Vert u_1-u_2\Vert_\sU \quad\text{for all }t\in [0,1].
\end{equation*}
To motivate our approach, assume that $J$ is twice continously differentiable. In this case, a possible approximation to the constant $C_n$ in iteration $n$ is $\Vert\nabla^2 J(u_n)\Vert$. Therefore, utilizing the previous gradient approximations, we obtain
\begin{equation*}
   C_n\approx \left\Vert \nabla^2 J(u_n)\right\Vert \approx \frac{\left\Vert \nabla J(u_n)-\nabla J(u_{n-1})\right\Vert}{\left\Vert u_n-u_{n-1}\right\Vert_\sU} \approx \frac{\big\Vert \hat{G}_n-\hat{G}_{n-1}\big\Vert}{\left\Vert u_n-u_{n-1}\right\Vert_\sU}.
\end{equation*}
Then, $C_n^{-1}$ yields a good initial step size for our backtracking line search. To circumvent high oscillation of $C_n$, which may arise from the approximation errors of the terms involved, we project $C_n$ onto the interval $\left[C_{\text{min}},C_{\text{max}}\right]\subset\R$, where $0<C_{\text{min}}<C_{\text{max}}<\tfrac{2^{T+1}}{L}$, i.e.,
\begin{equation}
    C_n = \min\left\{ C_{\text{max}}\, ,\,\max\left\{ C_{\text{min}}\,,\,\tfrac{\left\Vert \hat{G}_n-\hat{G}_{n-1}\right\Vert}{\left\Vert u_n-u_{n-1}\right\Vert_\sU}\right\}\right\}. \label{eq:CnForSICSG} 
\end{equation}
If possible, $C_{\text{min}}$ and $C_{\text{max}}$ should be chosen according to information concerning $L$. However, tight bounds on these quantities are not needed, as long as $T$ is chosen large enough. The resulting SCIBL-CSG (\textbf{SC}aling \textbf{I}ndependent \textbf{B}acktracking \textbf{L}ine search) method is presented in \Cref{alg:SICSG}. Notice that SCIBL-CSG does not require any a priori choice of step sizes and yields the same convergence results as bCSG.
\begin{algorithm} 
\caption{SCIBL-CSG}
\label{alg:SICSG}
\begin{algorithmic}[0]
\State{Given $u_0\in\U$,}
\While{Termination condition not met}
\State{Sample objective function:\\
    $\qquad j_n := j(u_n,x_n)$}
\State{Sample gradient:\\
    $\qquad g_n := \nabla_1 j(u_n,x_n)$}
\State{Calculate weights $\alpha_k$}
\State{Calculate search direction:\\
    $\qquad \hat G_{n} := \sum_{k=1}^n  \alpha_k g_k \vphantom{\Big(}$}
\State{Compute objective function value approximation:\\
    $\qquad \hat J_{n} := \sum_{k=1}^n  \alpha_k j_k \vphantom{\Big(}$} 
\State{Calculate $C_n$ by \eqref{eq:CnForSICSG}.}    
\State{Calculate step size $\tau_n$ by \Cref{alg:Backtracking} with start at $\tfrac{1}{C_n}$.}
\State{Gradient step:\\
    $\qquad u_{n+1} := \mathcal{P}_{\U}\big(u_n - \tau_n \hat G_{n}\big) \vphantom{\Big(}$}
\State{Update index:\\
    $\qquad n = n+1 \vphantom{\Big(}$}
\EndWhile
\end{algorithmic}
\end{algorithm}
\FloatBarrier
\section{Conclusion and Outlook}
In this contribution, we provided a detailed convergence analysis of the CSG method. The calculation of the integration weights was enhanced by several new approaches, which have been discussed and generalized for the possible implementation of mini-batches. 

We provided a convergence proof for the CSG method when carried out with a small enough constant step size. Additionally, it was shown that CSG can be augmented by an Armijo-type backtracking line search, based on the gradient and objective function approximations generated by CSG in the course of the iterations. The resulting bCSG scheme was proven to converge under mild assumptions and was shown to yield stable results for a large spectrum of hyperparameters. Lastly, we combined a heuristic approach for approximating the Lipschitz constant of the gradient with bCSG to obtain a method that requires no a priori step size rule and almost no information about the optimization problem.

For all CSG variants, the stated convergence results are similar to convergence results for full gradient schemes, i.e., every accumulation point of the sequence of iterates is stationary and we have convergence in objective function values. Furthermore, as is the case for full gradient methods, if the optimization problem has only finitely many stationary points, the presented CSG variants produce a sequence which is guaranteed to converge to one of these stationary points.

However, none of the presented convergence results for CSG give any indication of the underlying rate of convergence. Furthermore, while the performance of all proposed CSG variants was tested on academic examples, it is important to analyze how they compare to algorithms from literature and commercial solvers, when used in real world applications.

Detailed numerical results concerning both of these aspects can be found in \cite{CSGPart2}.

\bmhead{Data Availability Statement}
The simulation datasets generated during the current study are available from the corresponding author on reasonable request.
\bmhead{Conflict of Interests}
The authors have no relevant financial or non-financial interests to disclose.
\bibliography{Literature}


\end{document}